\newif\ifnotarxiv
\notarxivfalse
\documentclass{amsart}
\newtheorem{theorem}{Theorem}
\newtheorem{lemma}{Lemma}
\newtheorem{example}{Example}
\newtheorem{corollary}{Corollary}
\author{Kristian Debrabant and Anne Kv{\ae}rn{\o} and Nicky Cordua Mattsson}
\address[Kristian Debrabant]{Department of Mathematics and Computer Science, University of Southern Denmark, 5230 Odense M, Denmark}
\email{debrabant@imada.sdu.dk}
\address[Anne Kv{\ae}rn{\o}]{Department of Mathematical Sciences, Norwegian University of Science and Technology, 7491 Trondheim, Norway}
\email{anne.kvarno@math.ntnu.no}
\address[Nicky Cordua Mattsson]{Department of Mathematics and Computer Science, University of Southern Denmark, 5230 Odense M, Denmark}
\email{nickymattsson@icloud.com}
\title[Lawson schemes for highly oscillatory SDEs]{Lawson schemes for highly oscillatory stochastic differential equations and conservation of invariants}
\newcommand{\subclass}[1]{\subjclass{#1}}
\newif\ifUpdateFigures\UpdateFigurestrue
\usepackage[numbers,sort&compress]{natbib}
\usepackage{newtxtext,newtxmath}
\usepackage{helvet}
\usepackage{courier}

\usepackage{graphicx}
\makeatletter
\long\def\@makecaption#1#2{%
  \vskip\abovecaptionskip
  \sbox\@tempboxa{#1: #2}%
  \ifdim \wd\@tempboxa >\hsize
    #1: #2\par
  \else
    \global \@minipagefalse
    \hb@xt@\hsize{\hfil\box\@tempboxa\hfil}%
  \fi
  \vskip\belowcaptionskip}
\makeatother

\usepackage{subcaption}

\usepackage{float}

\usepackage{enumerate}

\usepackage{amsmath}
\usepackage{mathtools}
\usepackage{amsfonts}

\usepackage{algorithm}
\usepackage{algorithmic}

\usepackage{pgfplots}
\usepgfplotslibrary{fillbetween}
\pgfplotsset{compat=newest,every axis/.append style={legend style={font=\tiny}},filter discard warning=false}
\usepackage{pgfplotstable}
\ifUpdateFigures\else\renewcommand{\pgfplotstableread}[3][]{}\fi
\usepgfplotslibrary{external}
\ifUpdateFigures\else\fi

\usepackage{todonotes,soul,csquotes}

\makeatletter
\renewcommand{\todo}[2][]{\tikzexternaldisable \@todo[#1]{#2}\tikzexternalenable}
\makeatother

\makeatletter
\if@todonotes@disabled

\else

\fi
\makeatother

\usepackage{scrtime}

\usepackage{hyperref}

\usepackage[capitalize]{cleveref}

\newtheorem{assumption}{Assumption}

\Crefname{assumption}{Assumption}{Assumptions}
\Crefname{example}{Example}{Examples}
\crefname{SDE}{the SDE}{SDEs}
\crefname{equation}{}{}
\creflabelformat{SDE}{(#2#1#3)}

\newcommand{\gh}{{\hat{g}}}
\newcommand{\gb}{\bar{g}}

\newcommand{\It}{\tilde{I}}
\newcommand{\R}{\mathbb{R}}
\newcommand{\N}{\mathbb{N}}

\newcommand{\dmath}{\mathrm{d}}

\newcommand{\dt}{\dmath t}

\newcommand{\dW}{\dmath W}
\newcommand{\dX}{\dmath X}
\newcommand{\dXlin}{\dmath \Xlin}
\newcommand{\dV}{\dmath V}

\newcommand{\I}{\mathcal{I}}

\DeclareMathOperator{\E}{E}
\newcommand{\tspec}{\tilde{t}}
\newcommand{\Xlin}{\hat{X}}

\newcommand\LegendImage[1]{
\draw[%
                /pgfplots/mesh=false,%
                bar width=3pt,%
                bar shift=0pt,%
                mark repeat=2,%
                mark phase=2,#1]
                plot coordinates {
                    (0cm,0cm)
                    (0.3cm,0cm)
                    (0.6cm,0cm)%
                };
}

\newenvironment{customlegend}[1][]{%
        \begingroup
        \csname pgfplots@init@cleared@structures\endcsname
        \pgfplotsset{#1}%
    }{%
        \csname pgfplots@createlegend\endcsname
        \endgroup
    }%
\def\addlegendimage{\csname pgfplots@addlegendimage\endcsname}

\pgfplotsset{TDSLstyle/.style={mark=o,red}}
\pgfplotsset{TFSLstyle/.style={mark=square,green}}
\pgfplotsset{MDSLstyle/.style={mark=x,brown}}
\pgfplotsset{MFSLstyle/.style={mark=triangle,blue}}
\pgfplotsset{Mstyle/.style={mark=diamond}}

\pgfplotsset{
    legend image with MDSL MFSL Midpoint overlaid/.style={
        legend image code/.code={%
            \LegendImage{/pgfplots/MFSLstyle}
            \LegendImage{/pgfplots/MDSLstyle}
             \LegendImage{/pgfplots/Mstyle}
        }
    },
}

\begin{document}
\captionsetup{width=0.9\linewidth}

\ifnotarxiv\maketitle\fi
\begin{abstract}
{In this paper, we consider a class of stochastic midpoint and trapezoidal Lawson schemes for the numerical discretization of highly oscillatory stochastic differential equations. These Lawson schemes incorporate both the linear drift and diffusion terms in the exponential operator. We prove that the midpoint Lawson schemes preserve quadratic invariants and discuss this property as well for the trapezoidal Lawson scheme. Numerical experiments demonstrate that the integration error for highly oscillatory problems is smaller than {that} of some standard methods.
}
\subclass{60-08 \and 65C30 \and 60H10 \and 34C15 \and 60H35}
\end{abstract}

\maketitle

\section{Introduction}\label{sec:intro}

In this paper, we consider Stratonovich stochastic differential equations (SDEs) in which the drift and
diffusion terms can be split into  linear and  non-linear parts,
\begin{equation}\label[SDE]{eq:SDEOrig}
  \dX (t) =\left(A_0X(t) + g_0(X(t))\right)\dt+ \sum_{m=1}^M \left(A_mX(t) +
  g_m(X(t))\right)\circ\dW_m(t), ~X(t_0) =X_{0},
\end{equation}
where $t\in I$, $W_m(t)$, $m=1,\dotsc,M$, denote independent, one-dimensional Wiener processes, and
the SDE is solved on the interval $I=[t_0,T]$. We assume that \cref{eq:SDEOrig} has a unique solution for $X_0\in\mathbb{R}^{d}$ and that $g_m\in C^{1}(\mathbb{R}^d,\mathbb{R}^d)$.
To simplify the notation, we define $W_0(t)=t$, so that \labelcref{eq:SDEOrig} can be written as
\begin{equation} \label[SDE]{equ:sde}
  \dX (t) = \sum_{m=0}^M \left(A_mX(t) +
  g_m(X(t))\right)\circ\dW_m(t), \quad t\in I, \quad
  X(t_0)=X_{0}.
\end{equation}
We will also assume that the matrices $A_m \in \mathbb{R}^{d\times d}$, $m=0,\dotsc,M$, are constant, and moreover chosen such that the following commutativity assumption is
satisfied:
\begin{assumption}[Commutativity]\label{ass:commute}
  \[ [A_l,A_k] =A_lA_k - A_kA_l = 0  \qquad \text{for all} \qquad l,k = 0,1,\dots,M. \]
\end{assumption}%
To satisfy this assumption, it is sometimes convenient to split the linear parts of the problem, and
let some of it be included in the $g_m$ functions.

{Applications satisfying \cref{ass:commute} are, e.\,g., \cite{erdogan19anc} the FitzHugh–Nagumo equation with multiplicative noise, the Lotka-Volterra system and SDEs resulting from spectral spatial discretization of stochastic partial differential equations (SPDEs) with diagonal noise.}

Exponential methods for solving such problems have in particular been applied in the
SPDE setting, mostly, but not exclusively for problems with
additive noise. \citet{cohen12otn} proposed an exponential method for stochastic oscillators, \citet{yang19sps} suggested one for damped Hamiltonian systems.
 Recently, \citet{erdogan19anc} presented a quite general approach for constructing exponential integrators for
\labelcref{equ:sde} with multiplicative noise. One of the strategies presented there is an adaptation of
a method introduced for ordinary differential equations (ODEs) by \citet{lawson67grk} to SDEs
of the form \labelcref{equ:sde}. The idea is to transform the system by the fundamental solution of
the linear part, solve the transformed system by a scheme of preference, and then transform back
again. In the ODE literature, this is also referred to as an integrating factor method (\citet{maday90aoi,cox02etd}).
This is the procedure which will be applied in this paper, and which is described in detail in \cref{sec:Lawsons}. We are essentially interested in studying highly oscillatory problems, and to show that the
Lawson methods can attain good accuracy with larger step sizes than what can be obtained by standard
stochastic methods. We will also show that the Lawson methods can, under reasonable
assumptions, maintain conservation properties of the underlying scheme. Let us demonstrate the ideas
of the paper by an introductory example.

\begin{example}[The non-linear Kubo oscillator]\label{ex:nonlinkubo}
  As a starting point, consider the linear Kubo oscillator
    \begin{equation}
    \dX(t) = \begin{pmatrix}
      0 & -\omega \\ \omega & 0
    \end{pmatrix} X(t) \dt + \begin{pmatrix}
      0 & -\sigma \\ \sigma & 0
    \end{pmatrix} X(t) \circ \dW(t)
  \end{equation}
  which models a simple oscillator perturbed by a stochastic term and appears in nuclear magnetic
  resonance and molecular spectroscopy (\citet{kubo63sle,goychuk04qdw}). The exact solution of this problem starting at $(\tspec,X(\tspec))$ for $\tspec \in I$ is given by
  \[
    X(t) = e^{L^{\tspec}(t)}X(\tspec),
  \]
  where the fundamental solution $e^{L^{\tspec}(t)}$ is a rotation matrix,
  \[
    e^{L^{\tspec}(t)} =
    \begin{pmatrix}
      \cos{\alpha^{\tspec}(t)} & -\sin{\alpha^{\tspec}(t)} \\ \sin{\alpha^{\tspec}(t)} & \cos{\alpha^{\tspec}(t)}
    \end{pmatrix}
    \text{ with } \alpha^{\tspec}(t) = \omega (t-\tspec)+\sigma (W(t)-W(\tspec)).
  \]
  Clearly, the exact solution $X(t)=(X_1(t),X_2(t))^\top$ of the linear problem is norm-preserving, i.\,e.\
  satisfies the invariant
  \begin{equation} \label{equ:KuboInvar}
    \I (X(t)) = X_1^2(t)+X_2^2(t) = \text{constant} \quad \text{for all } t.
  \end{equation}
  \citet{cohen12otn} proposed an extension to the Kubo oscillator by including a non-linear,
  skew-symmetric drift term, see also \citet{laurent20mif}.
  We extend this further, with non-linear terms in both the drift and diffusion, in addition to
  including multidimensional noise:
  \begin{equation}\label{equ:NonLinKubo}
    \dX(t) = \sum_{m=0}^M \left[ \omega_m \begin{pmatrix}
	0 & -1 \\ 1 & 0
      \end{pmatrix} X + \begin{pmatrix}
	0 & -U_m(X) \\ U_m(X) & 0
    \end{pmatrix} X(t) \right] \circ\dW_m(t)
  \end{equation}
  with $U_m:\mathbb{R}^2 \to \mathbb{R}$ and $\omega_m \in \mathbb{R}$. The solutions of these
  problems all preserve the invariant \eqref{equ:KuboInvar}, see \cref{sec:QuadraticInv} for details.
  We are interested in studying to which {extent} a numerical approximation will be
  able to follow the fast oscillations of the linear parts, as well as how well the
  invariant \eqref{equ:KuboInvar} is preserved. In this example, the following three methods (described in
  detail in \cref{sec:Lawsons}) have been applied to the SDE \eqref{equ:NonLinKubo}:

      \begin{itemize}
    \item The standard implicit stochastic midpoint rule (\enquote{Midpoint}), which is known to preserve quadratic invariants (\citet{milstein02nmf,hong15poq}).
    \item The method proposed by \citet{cohen12otn} (\enquote{TDSL}) for highly oscillatory SDEs.
      This is a drift Lawson scheme based on the trapezoidal rule,
      but it does not preserve the invariant for the non-linear problem.
    \item A Lawson scheme based on the implicit midpoint rule (\enquote{MFSL}). Details of this scheme are given in \cref{sec:Lawsons},
      and in \cref{sec:QuadraticInv} it is proved that this scheme preserves the quadratic invariant $\I$.
  \end{itemize}

  In our example we will use $M=2$ and
  \begin{equation*}
    U_0(X) = \frac{1}{5}(X_1 + X_2)^5, ~ U_1(X)=0, ~ U_2(X)=\frac{1}{3}(X_1+X_2)^3,~ \omega_0=
    10, ~ \omega_1=10, ~ \omega_2=0,
  \end{equation*}
  and the SDE is integrated from 0 to 1, using step size $h=2^{-5}$.

\pgfplotstableread[col sep=comma]{Data/Intro/RefOmega10Sigma10Start.csv}\refF
\pgfplotstableread[col sep=comma]{Data/Intro/h2-4Omega10Sigma10Start.csv}\intro

\ifUpdateFigures
\pgfplotstablegetrowsof{\refF}
\pgfmathtruncatemacro{\rownumL}{\pgfplotsretval}

\pgfplotstablegetrowsof{\intro}
\pgfmathtruncatemacro{\rownumS}{\pgfplotsretval}
\fi
\begin{figure}[ht]
\begin{minipage}{0.3333\linewidth}
\begin{tikzpicture}
\begin{axis}[%
	name=1,
	clip=true,
	clip bounding box=upper bound,
	hide axis,
    width=\textwidth,
    height=0.3\textheight,
    xmin=-1.1, xmax=1.1,
    ymin=-1.1, ymax=1.1,
    title={Midpoint}
    ]

    \addplot3[%
    opacity = 0.02,
    fill opacity=0.5,
    mesh/interior colormap name=hot,
    surf,
    colormap/hot,
    faceted color=black,
    z buffer = sort,
    samples = 50,
    variable = \u,	
    variable y = \v,
    domain = 0:360,
    y domain = 0:1,
    ]
    ({cos(u)}, {sin(u)}, {v});

    \addplot3[mark=o,blue,mark indices={\rownumL},mark options={draw=blue, ultra thick, scale=2, fill}] table[x={MFSL1}, y ={MFSL2}, z ={t}]{\refF};
    \addplot3[only marks,mark=triangle*,red] table[x={Midpoint1}, y ={Midpoint2}, z ={t}]{\intro};
	\addplot3[only marks,mark indices={\rownumS},mark=o,mark options={draw=red, ultra thick, scale=2, fill}] table[x={Midpoint1}, y ={Midpoint2}, z ={t}]{\intro};
    \end{axis}
    \draw[|->] (3,0.25) -- (3,4.5) node[pos=1,right]{$t$};
\end{tikzpicture}
\end{minipage}%
\begin{minipage}{0.3333\linewidth}
\begin{tikzpicture}
\begin{axis}[%
	name=plot1,
	clip=true,
	clip bounding box=upper bound,
	hide axis,
    width=\textwidth,
    height=0.3\textheight,
    title={TDSL}
    ]

	\addplot3[%
    opacity = 0.02,
    fill opacity=0.5,
    mesh/interior colormap name=hot,
    surf,
    colormap/hot,
    faceted color=black,
    z buffer = sort,
    samples = 50,
    variable = \u,
    variable y = \v,
    domain = 0:360,
    y domain = 0:1,
    ]
    ({cos(u)}, {sin(u)}, {v});

    \addplot3[mark=o,blue,mark indices={\rownumL},mark options={draw=blue,ultra thick, scale=2, fill}] table[x={MFSL1}, y ={MFSL2}, z ={t}]{\refF};
    \addplot3[only marks,mark=triangle*,red] table[x={TDSL1}, y ={TDSL2}, z ={t}]{\intro};
    \addplot3[only marks,mark indices={\rownumS},mark=o,mark options={draw=red,ultra thick, scale=2, fill}] table[x={TDSL1}, y ={TDSL2}, z ={t}]{\intro};

    \end{axis}
    \draw[|->] (3,0.25) -- (3,4.5) node[pos=1,right]{$t$};
\end{tikzpicture}	
\end{minipage}%
\begin{minipage}{0.3333\linewidth}
	\begin{tikzpicture}
\begin{axis}[%
	name = plot1,
	clip=true,
	clip bounding box=upper bound,
	hide axis,
    width=\textwidth,
    height=0.3\textheight,
    xmin=-1.1, xmax=1.1,
    ymin=-1.1, ymax=1.1,
    title={MFSL}
    ]

	\addplot3[%
    opacity = 0.02,
    fill opacity=0.5,
    mesh/interior colormap name=hot,
    surf,
    colormap/hot,
    faceted color=black,
    z buffer = sort,
    samples = 50,
    variable = \u,
    variable y = \v,
    domain = 0:360,
    y domain = 0:1,
    ]
    ({cos(u)}, {sin(u)}, {v});

    \addplot3[mark=o,blue,mark indices={\rownumL},mark options={draw=blue, ultra thick, scale=2, fill}] table[x={MFSL1}, y ={MFSL2}, z ={t}]{\refF};
    \addplot3[only marks,mark=triangle*,red] table[x={MFSL1}, y ={MFSL2}, z ={t}]{\intro};
	\addplot3[only marks,mark indices={\rownumS},mark=o,mark options={draw=red, ultra thick, scale=2, fill}] table[x={MFSL1}, y ={MFSL2}, z ={t}]{\intro};
    \end{axis}
    \draw[|->] (3,0.25) -- (3,4.5) node[pos=1,right]{$t$};
\end{tikzpicture}
\end{minipage}%
\caption{Numerical trajectory of the non-linear Kubo oscillator. Blue: reference solution, blue circle: value at $t=1$. Red triangles: numerical approximation, red circle: value at $t=1$.\label{fig:IntroExample}}%
\end{figure}

The numerical results are presented in \cref{fig:IntroExample}. According to this, the implicit midpoint rule (\enquote{Midpoint}) preserves the invariant,
  it is however not able to resolve the fast oscillations. The method proposed by \citet{cohen12otn} (\enquote{TDSL}) resolves the oscillations well, but the solution drifts away from the manifold. The Lawson midpoint rule (\enquote{MFSL}) both stays on the manifold and resolves the high
  oscillations.
  This example is further exploited in \cref{sec:numEx}.
\end{example}
The outline of this paper is as follows: In \cref{sec:Lawsons} we derive the midpoint and trapezoidal stochastic Lawson schemes.
In \cref{sec:PresInv} it is proved that, under reasonable assumptions, the Lawson transformation preserves linear and
quadratic invariants, and consequently, if the underlying scheme preserves such invariants, so will the
corresponding Lawson scheme. \Cref{sec:numEx} is devoted to numerical experiments.

\section{Stochastic Lawson schemes} \label{sec:Lawsons}
In this section we will shortly outline how to derive the stochastic Lawson (SL) schemes, with particular emphasis
on the implicit midpoint and the trapezoidal Lawson scheme.

Consider the linear SDE
\begin{equation}\label[SDE]{equ:linSDE}
  \dXlin(t) = \sum_{m=0}^M A_m \Xlin(t) \circ \dW_m(t)
\end{equation}
and let
\[W^{\tspec}(t) = ( W_m(t)-W_m(\tspec))_{m=0}^M\qquad\text{ and }\qquad L{(W^{\tspec}(t))} = \sum_{m=0}^M A_m W_m^{\tspec}(t).\]
Under the
commutativity condition (\cref{ass:commute}), $e^{L{(W^{\tspec}(t))}}$ and $A_m$ commute for $m=0,\dotsc,M$,
and the exact solution of \cref{equ:linSDE} through the
point $(\tspec,\Xlin(\tspec))$ is given by {(see }\citet[Chapter 8.5]{arnold74sde}{)}
\begin{equation}
  \Xlin(t) = e^{L{(W^{\tspec}(t))}}\Xlin(\tspec).
  \label{equ:linsde}
\end{equation}
This can easily be verified by the chain rule for Stratonovich integrals {(see e.g.} \citet[Chapter III,
Theorem 2.1]{ikeda89sde}{)},
\begin{align*}
\dXlin(t)&=e^{L{(W^{\tspec}(t))}}\sum_{m=0}^M A_m\circ\dW_m(t)\Xlin(\tspec)\\
& =\sum_{m=0}^M A_me^{L{(W^{\tspec}(t))}}\Xlin(\tspec)\circ\dW_m(t)=\sum_{m=0}^M A_m\Xlin(t)\circ\dW_m(t).
\end{align*}
We will now outline the procedure for constructing Lawson schemes. Let discretization points $t_0<t_1<\dots<t_N=T$ be given and denote the approximations by the Lawson scheme  at time $t_n$ by $Y_n$, $n=0,\dots,N$.
Start from a point $(t_n,Y_n)$ and consider the locally transformed variable
\begin{equation}
\label{equ:TransformedVar}
V^n(t) = e^{-L^{n}(t)} X(t) \qquad \text{and thus} \qquad X(t) = e^{L^n(t)}V^n(t),
\end{equation}
where $L^n(t)=L{{(W^{t_n}(t))}}$.
Applying again the chain rule for Stratonovich integrals results under \cref{ass:commute} in
\begin{align*}
  \dV^n(t)
  &=e^{-L^{n}(t)}\left(-\sum_{m=0}^{M}A_m\circ\dW_m(t)X(t)+\sum_{m=0}^M \left(A_mX(t) +
  g_m(X(t))\right)\circ\dW_m(t) \right)\\
  &= e^{-L^n(t)}\sum_{m=0}^M g_m(X)\circ\dW_m(t),
\end{align*}
which can be written as a non-autonomous, non-linear SDE in the transformed variable $V^n(t)$,
\begin{align}
\dV^n(t) &= \sum_{m=0}^M e^{-L^n(t)}g_m(e^{L^n(t)}V^{n}(t)) \circ \dW_m(t) \nonumber \\
         &=\sum_{m=0}^M\gh_m(W^n(t),V^n(t)) \circ\dW_m(t), \qquad  V^n(t_n)= V^n_n = Y_n,
  \label[SDE]{equ:TransformedSDE}
\end{align}
where $W^n = W^{t_n}$ and $\gh_m({y},V^n(t))=e^{-L{(y)}}g_m(e^{L{(y)}}V^{n}(t))$. One step from $(t_n,Y_n)$ to $(t_{n+1},Y_{n+1})$ with a Lawson scheme is now just one step of some
appropriate one-step method applied to {the autonomous version
\begin{equation}\label{equ:autonomousSDE}
    \dmath \bar{V}^n(t) = \sum_{m=0}^M \gb_m(\bar{V}^n(t)) \star \dW_m(t), \quad \gb_m = (\delta_{0,m},\ldots,\delta_{M,m},{\gh^{\top}_m})^{\top}
\end{equation}
of }the transformed system \labelcref{equ:TransformedSDE}{, with $\bar{V}^n(t) = (W^n(t)^\top,{V^n}(t)^\top)^{\top}$ and $\delta_{i,j} = \begin{cases}
        1, & i=j, \\
        0, & i\neq j,
    \end{cases}$} giving $V^n_{n+1}$, followed by a back-transformation
$Y_{n+1}=e^{L^n(t_{n+1})}V^n_{n+1}$ (see \citet{debrabant21rkl} for details).
Due to the construction of the methods, the convergence properties of the underlying methods are
retained (\citet{debrabant21rkl}). {Specifically, we have the following theorem:
\begin{theorem}[Convergence of stochastic Lawson methods \cite{debrabant21rkl}]\label{thm:convergenceLawsonScheme}
Let \cref{ass:commute} hold, let $X$ be the solution of SDE \eqref{equ:sde}, $Y_n$ be the result of the stochastic Lawson method ($n=0,\dots,N$), $\bar{V}^0$ be the exact solution of \cref{equ:autonomousSDE} (with $n=0$), and $\bar{V}^0_{n}=({W^0_n}^\top,{V^0_n}^\top)^{\top}$ for $n=0,\dots,N$ be its approximation obtained by applying the underlying SRK method {with step sizes $h_i=t_{i}-t_{i-1}$, $i=1,\dots,N$. Further, let $h_{N,max}=\max_{i=1}^Nh_i$}.
\begin{enumerate}
\item Assume that $W^0_n=W^0(t_n)$ and that the underlying SRK method is of mean square order $p$, i.\,e.\ there exists a $c\in \mathbb{R}$ such that for all $N\in\N$ and all $n\in\{0,1,\dots,N\}$ it holds that $\sqrt{{\E}( \|{V}^0_n - {V}^0(t_n)\|_2^2)} \leq ch_{{N,max}}^p$.
Then the stochastic Lawson method is strong convergent of order $p$, i.\,e., there exists a $\tilde{c} \in \mathbb{R}$ such that for all $N\in\N$ and all $n\in\{0,1,\dots,N\}$ it holds {that}
    \begin{equation}
        {\E}\|Y_n - X(t_n)\|_2 \leq \tilde{c}h_{{N,max}}^p.
    \end{equation}
\item
Assume that
\begin{enumerate}[(a)]
\item $A_m$ for $m> 0$ are skew-symmetric (cmp.\ \cref{ass:ABskew} below),
\item $\bar{V}^0_n$ is of weak order $\tilde{p}$, i.e.\ for all $\tilde{f}\in C^{2(\tilde{p}+1)}_P(\R^{M+1}\times\R^{d},\R)$ there exists a $c\in \mathbb{R}$ such that for all $N\in\N$ and all $n\in\{0,1,\dots,N\}$ it holds that $|{\E}(\tilde{f}(\bar{V}^0_n) - \tilde{f}(\bar{V}^0(t_n))| \leq ch_{{N,max}}^{\tilde{p}}$,
\item $W^0_{n,0}=W_0(t_n)=t_n$.
\end{enumerate}
Then the stochastic Lawson method is of weak order $\tilde{p}$, i.\,e., for all $f\in C^{2(\tilde{p}+1)}_P{(\R^{d},\R)}$ there exists some constant $c_f > 0$ such that for all $N\in\N$ and all $n\in\{0,1,\dots,N\}$ it holds
\begin{equation}
         | {\E} f(Y_{{n}}) - {\E}f(X(t_{{n}}))| \leq c_{{f}} h_{{N,max}}^{\tilde{p}}.
\end{equation}
\end{enumerate}
\end{theorem}
}

In this paper, only Lawson schemes based on the implicit trapezoidal and midpoint rules will
be discussed. These two methods, applied to {\cref{equ:autonomousSDE}}, are given by
\begin{align}
  V_{n+1}^n &= V^n_n + \sum_{m=0}^M \bigg(\gh_m({0},V^n_n)+
  \gh_m(\Delta W^n,V^n_{n+1})\bigg) \frac{\Delta W^n_m}{2}, \label{eq:Trapezoidalrule} \\
  \intertext{respectively}
  V^n_{n+1} &= V^n_n + \sum_{m=0}^M\label{eq:Midpointrule}
    \gh_m\left({\frac{\Delta W^n}{2}},\frac{V^n_{n}+V^n_{n+1}}{2} \right) \Delta
    W^n_m,
\end{align}
where the stochastic increments are $\Delta W_m^n = W_m(t_{n+1})-W_m(t_n)$, $m=0,\dotsc,M$, {and we used that $W^n(t_n)=0\in\R^{M+1}$ and $W^n(t_{n+1})=\Delta W^n$.}
{Under appropriate conditions on the smoothness and boundedness of the $\gh_m$ (e.\,g.\ for mean square convergence of order up to one that $\gh_0+\frac12\sum_{m=1}^M\gh_m'\gh_m$ and $\gh_1,\dots,\gh_M$ satisfy a global Lipschitz condition and are, together with all the associated elementary differentials up to order three, of linear growth), these approximations will be of mean square order $p=0.5$ ($p=1$ for commutative noise) and weak order $\tilde{p}=1$, see e.\,g.\ \citet{milstein04snf,kloeden99nso,hong15poq,debrabant08bsa}. Note that for methods that conserve a bounded manifold, smooth functions $g$ can be replaced by smooth bounded functions being zero outside a suitable ball, see e.g.\ \citet{cohen20hon}.}

Applying the back-transformation results in the two Lawson methods of
interest in this paper:
\begin{align}
  \intertext{Trapezoidal Lawson rule:}
Y_{n+1} &= e^{\Delta L^n}Y_n + \sum_{m=0}^M \bigg(e^{\Delta L^n}g_m(Y_n)+g_m(Y_{n+1}) \bigg)
\frac{\Delta W_m^n}{2}, \label{equ:LawsonTrap} \\
  \intertext{Midpoint Lawson rule:}
  Y_{n+1} &= e^{\Delta L^n}Y_n + \sum_{m=0}^M e^{\frac{\Delta L^n}{2}} g_m\left(
  \frac{e^{\frac{\Delta L^n}{2}}Y_n + e^{-\frac{\Delta L^n}{2}}Y_{n+1}}{2} \right) \Delta W^n_m,  \label{equ:LawsonMid}
\end{align}
where $\Delta L^n = L^n(t_{n+1}) = \sum_{m=0}^M A_m \Delta W_m^n$.

We would like to emphasize that there is some freedom in how to choose the matrices $A_m$ in \cref{equ:sde} and thus the Lawson methods \eqref{equ:LawsonTrap} and \eqref{equ:LawsonMid}, as for any $A_m\in \mathbb{R}^{d\times d}$ an integrand $\tilde{A}_mX + \tilde{g}_m(X)$ can be written as
\[ \tilde{A}_mX + \tilde{g}_m(X) = A_mX +  \overbrace{(\tilde{A}_m-A_m)X+\tilde{g}_m(X)}^{=:g_m(X)},
\]
see also \citet{debrabant21rkl}. The matrices $A_m$ have to be chosen such that \cref{ass:commute} is satisfied.  Different splittings lead to different schemes. The underlying
numerical scheme is restored by setting all the linear parts $A_m$ to 0 ($m=0,1,\dotsc,M$).
In a drift stochastic Lawson scheme (DSL), only the linear drift term is included in the
exponential, that is $A_0\not=0$ and $A_m=0$ for $m=1,\dotsc,M$. The
exponential $e^{\Delta L^n}=e^{h_{{n+1}}A_0}$ is computed only once, and the DSL schemes can be quite
efficient. The scheme suggested by \citet{cohen12otn} (denoted by TDSL in the present article) is an example of a DSL
scheme based on the trapezoidal rule \eqref{equ:LawsonTrap}. If  $A_m\not=0$ for at least one
$m\not=0$, the scheme is denoted as a full stochastic Lawson scheme (FSL). In this case the exponentials depend
on the stochastic increments and have to be calculated for each step, so the performance of the FSL schemes depends on how efficient this can be done.

For a continuous splitting between the linear and the non-linear part, see \citet{erdogan19anc}.

\section{Preservation of invariants} \label{sec:PresInv}
{\citet{hong15poq})} studied preservation of linear and quadratic invariants for SDEs.
In particular, they showed that the stochastic midpoint rule (as a representative for the Gauss methods),
preserves quadratic invariants. The aim of this section is to extend these results to
the stochastic Lawson (SL) schemes. It is also well known that the standard trapezoidal rule almost
preserves quadratic invariants for deterministic differential equations, see e.g.\  \citet[Example
V.4.2]{hairer06gni}. This property does unfortunately not extend to the stochastic trapezoidal rule in general; we will study under which conditions it does hold.
The commutativity condition \cref{ass:commute} is always assumed {to be} satisfied throughout this section.

In the following, we first investigate when stochastic Lawson schemes preserve quadratic invariants before briefly commenting on the preservation of linear invariants.
\subsection{Preservation of quadratic invariants} \label{sec:QuadraticInv}
In this section we study when \cref{equ:sde} for a $D\in \mathbb{R}^{d\times d}$ satisfies the quadratic invariant
\begin{equation}\label{eq:quadraticinvariant}
    \I(X(t)) := X(t)^{\top} D X(t) = \I(X(t_0)),
\end{equation}
 and the ability of the numerical schemes to preserve it. Without loss of generality we will assume that $D=D^{\top}$. For the rest of this section we will use a combination of the following assumptions:

\begin{assumption} \label{ass:ABskew}
	The matrices $A_m\in \mathbb{R}^{d\times d}$, $m=0,\dots,M$, are skew-symmetric and commute with $D$.
\end{assumption}
\begin{assumption} \label{ass:SkewSym}
    For all $X\in\mathbb{R}^d$ and $m=0,\dots,M$ we have $X^{\top} D g_m(X)=0$.
\end{assumption}
The necessity of these two assumptions can be seen by applying the chain rule for Stratonovich integrals and using the symmetry of $D$ to obtain
    \begin{align*}
      \dmath \I(X(t)) = 2\sum_{m=0}^{M} \bigg(X(t)^{\top}D A_m X + X(t)^{\top} D g_m(X)\bigg) \circ \dW_m.
\end{align*}
To preserve $\I(X(t))$ we need both terms to be zero, the first is zero if \cref{ass:ABskew} holds and the second is zero if \cref{ass:SkewSym} holds. In summary, it holds (see e.g. \citet{hong15poq}):
\begin{lemma}\label{lem:InvSDE}
    Under \cref{ass:SkewSym,ass:ABskew}, \cref{equ:sde} has a quadratic invariant, i.\,e.\ its solution $X(t)$ fulfills \cref{eq:quadraticinvariant}.
\end{lemma}
We now prove that the invariant \eqref{eq:quadraticinvariant} is preserved under transformations \eqref{equ:TransformedVar}:
\begin{lemma}\label{lem:TransfPresInv}
    Let \cref{ass:ABskew} be fulfilled. Then it holds for $\I$ defined in \cref{eq:quadraticinvariant}  that
   \begin{equation}\label{eq:TransfPresInv}
     \I(v) = \I(e^{L^n(t)}v), \qquad \text{for all } v \in \R^d.
    \end{equation}
\end{lemma}
\begin{proof}
By \eqref{eq:quadraticinvariant} and using that $A_m$, $m=0,\dots,M$, are skew-symmetric (by \cref{ass:ABskew}) and $e^{L^n(t)}$ thus orthogonal, we obtain
\begin{align*}
    \I(e^{L^n(t)}v)= v^{\top} e^{-L^n(t)} D e^{L^n(t)} v
\end{align*}
which, using that all $A_m$ commute with $D$ (by \cref{ass:ABskew}) simplifies to
\begin{align*}
    \I(e^{L^n(t)}v)= v^{\top} D v = \I(v).
\end{align*}
\end{proof}

We can then prove that the transformed system preserves the same invariant as the original one:
\begin{lemma}\label{lem:InvTransSDE}
    Let \cref{ass:commute,ass:SkewSym,ass:ABskew} hold. Then the transformed system \eqref{equ:TransformedSDE} with solution $V^n(t)$ preserves the same invariant as the original system \eqref{equ:sde} with solution $X(t)$, i.e.\  \begin{equation}\label{eq:InvTransfSDE}
        \I(V^n(t)) = \I(X(t))=\I(X(t_0)).
    \end{equation}
\end{lemma}

\begin{proof}
By \cref{equ:TransformedVar} it holds for $V^n(t)$ that (under \cref{ass:commute})
\begin{align*}
    \I(X(t)) = \I(e^{L^n(t)}V^n(t)).
\end{align*}
\cref{lem:TransfPresInv} implies then $\I(X(t)) = \I(V^n(t)).$
By \cref{lem:InvSDE} we can conclude {that} \cref{eq:InvTransfSDE} {holds}.
\end{proof}
We are now ready to state the main result of this section:
\begin{theorem}\label{thm:NumQuadInv}
Let \cref{ass:commute,ass:SkewSym,ass:ABskew} hold and let $Y_n$ denote the discrete time approximation of \cref{equ:sde} at time point $t_n$ using an SL scheme with an underlying numerical one-step method that preserves quadratic invariants. Then the SL scheme preserves the same quadratic invariant as \cref{equ:sde}, i.e.\  \begin{equation}
        \I(Y_n) = \I(X(t))=\I(X(t_0)).
    \end{equation}
\end{theorem}

Requirements for stochastic Runge--Kutta methods to preserve quadratic invariants are given in
\citet{hong15poq}. Strictly speaking, those results only apply to autonomous systems, it is however
straightforward to extend them to the nonautonomous transformed
SDE \eqref{equ:TransformedSDE}. We can then conclude with the following \lcnamecref{cor:QInvMidpoint}:

\begin{corollary}\label{cor:QInvMidpoint}
  Under the assumptions of \cref{thm:NumQuadInv} the midpoint stochastic Lawson scheme
  preserves the quadratic invariants \labelcref{eq:quadraticinvariant}.
\end{corollary}

\begin{proof}[of \cref{thm:NumQuadInv}]
This follows directly by applying a numerical one-step method that preserves quadratic invariants to the transformed system:
Let $V^{n}_{n+1}$ be the numerical approximation at time $t_{n+1}$ of \cref{equ:TransformedSDE}. By \cref{lem:InvTransSDE} and the scheme preserving quadratic invariants, it holds that
\[
\I(V^{n}_{n+1})=\I(V^n(t_n))=\I(X(t))=\I(X(t_0)).
\]
By \cref{lem:TransfPresInv} and the definition of $Y_{n+1} = e^{\Delta L^n}V^n_{n+1}$ it follows that
\[
\I(Y_{n+1})=\I(V^{n}_{n+1})=\I(X(t))=\I(X(t_0)),
\]
which finishes the proof.
\end{proof}

It is well known that the standard trapezoidal rule almost preserves quadratic invariants for
deterministic differential equations, see e.g.\  \citet[Example V.4.2]{hairer06gni}. This property
does unfortunately not extend to the stochastic trapezoidal rule in general. But if all the drift and diffusion
terms are linear and included in the $L$ operator, a similar result can be obtained, which is also
a good argument for using full stochastic Lawson schemes instead of the drift ones.

\begin{theorem}\label{cor:TrapzruleQinvariant}
Let \cref{ass:commute,ass:SkewSym,ass:ABskew} hold, let $g_m\equiv0$ for $m=1,\dots,M$ in \cref{equ:sde} (all diffusion terms are linear and commute) and
let $Y_n$ denote its discrete time approximation using the trapezoidal FSL scheme {with equidistant step sizes $h$}. Then
\begin{equation}
  \I(Y_{n}) - \I(Y_0) = -\frac14 \left( g_0(Y_n)^{\top}Dg_0(Y_n)-  g_0(Y_0)^{\top}Dg_0(Y_0)
  \right)h^2.
\end{equation}
\end{theorem}
\begin{proof}
The trapezoidal Lawson rule \eqref{equ:LawsonTrap} can be considered as composed by two half steps,
one with an exponential forward Euler step, and one with an exponential backward Euler step, as
follows:
\begin{align}
  \hat{Y}_n &= e^{\frac12 \Delta L^n}Y_n + e^{\frac12 \Delta L^n} \label{equ:fe}
  \sum_{m=0}^{M}g_m(Y_n)\frac{\Delta W^n_m}{2}, \\
  Y_{n+1} &= e^{\frac12 \Delta L^n}\hat{Y}_n + \sum_{m=0}^{M}g_m(Y_{n+1})\frac{\Delta W^n_m}{2}.
  \label{equ:be}
\end{align}
Doing one more half step with the forward Euler method gives
\begin{align*}
  \hat{Y}_{n+1} &=  e^{\frac12 \Delta L^{n+1}}Y_{n+1} + e^{\frac12 \Delta L^{n+1}}
  \sum_{m=0}^{M}g_m(Y_{n+1})\frac{\Delta W^{n+1}_m}{2} \\
  &= e^{\frac12 (\Delta L^n+\Delta L^{n+1})} \hat{Y}_n + e^{\frac12 \Delta
  L^{n+1}}\sum_{m=0}^{M}g_m(Y_{n+1}) \frac{\Delta W_m^n + \Delta W_m^{n+1}}{2},
\end{align*}
which together with \eqref{equ:be} could be considered as one step of a modified version of the
midpoint rule starting from $\hat{Y}_n$, and we would like to see if it conserves quadratic
invariants: By \cref{ass:ABskew}, $e^{\Delta L^n}$ is orthogonal and commutes with $D$, so
\begin{multline}\label{eq:int}
  \hat{Y}_{n+1}^{\top}D\hat{Y}_{n+1} = \hat{Y}_n^{\top}D\hat{Y}_n
  + \sum_{m=0}^MY_{n+1}^{\top} D g_m(Y_{n+1}) (\Delta W_m^n+\Delta W_m^{n+1}) \\
  + \sum_{m,r=0}^M g_r(Y_{n+1})^{\top} D g_m(Y_{n+1}) \frac{(\Delta W_r^{n+1}-\Delta
  W_r^{n})({\Delta} W_m^{n+1}+{\Delta} W_m^{n})}{4},
\end{multline}
where we also used \cref{ass:commute} and that by \eqref{equ:be} it holds that \ifnotarxiv$\else\[\fi e^{\frac12 \Delta L^n}\hat{Y}_n = Y_{n+1} -
\sum_{r=0}^{M}g_r(Y_{n+1})\frac{\Delta W^n_r}{2}\ifnotarxiv$.\else.\]\fi
The second term of \eqref{eq:int} vanishes since $Y^{\top}Dg(Y)=0$ for all $Y\in\mathbb{R}^d$ by
\cref{ass:SkewSym}. Due to the last term,  the quadratic invariant is not preserved in general.
Since  $\Delta W_0=h$, if $g_m=0$ for $m=1,\dotsc M$ the last
sum is 0, and $\hat{Y}_{n+1}^{\top}D\hat{Y}_{n+1} = \hat{Y}_n^{\top}D\hat{Y}_n$, which by \eqref{eq:quadraticinvariant} can be written as
$\I(\hat{Y}_{n+1})=\I(\hat{Y}_{n})$. It follows that
\begin{align*}
  \I(Y_{n+1}) - \I(Y_0) & = \I(Y_{n+1}) - \I(\hat{Y}_{n})+\I(\hat{Y}_{n})-\I(\hat{Y}_0)
  + \I(\hat{Y}_0) - \I({Y}_0) \\
  & = \I(Y_{n+1}) - \I(\hat{Y}_{n}) +  \I(\hat{Y}_0) - \I({Y}_0).
\end{align*}
Direct computations show that
\begin{align*}
  \I(\hat{Y}_0)-\I(Y_0) &= \frac14 g_0(Y_0)^{\top}Dg_0(Y_0)h^2, \ifnotarxiv\qquad\else\ \fi \\
  \I(Y_{n+1}) - \I(\hat{Y}_{n}) &= -\frac14 g_0(Y_{n+1})^{\top}Dg_0(Y_{n+1})h^2,
\end{align*}
and in conclusion
\[
  \I(Y_{n}) - \I(Y_0) = -\frac14 \left( g_0(Y_n)^{\top}Dg_0(Y_n)-  g_0(Y_0)^{\top}Dg_0(Y_0)
  \right)h^2.
\]
\end{proof}

Notice that a similar argument can also be used as a direct proof of \cref{cor:QInvMidpoint}: One step of
the stochastic midpoint Lawson scheme is composed by one half step of the exponential backward Euler scheme
followed by one half step of the exponential forward Euler scheme, both using the same stochastic increments
$\Delta W^n_m/2$, so the last sum of \eqref{eq:int} is 0.

Notice also that if the quadratic invariant is associated to a bounded manifold, i.\,e., the eigenvalues of $D$ are all positive, then the proof of \cref{cor:TrapzruleQinvariant} implies that under the conditions given there, there exists a constant $c$ such that
\[
|\I(Y_{n}) - \I(Y_0)|\leq c h^2
\]
for all $n=0,1,\dots,N$ and $h$ sufficiently small.

\subsection{Preservation of linear invariants}
For completeness we also study when \cref{equ:sde} preserves linear invariants
\begin{equation}
	\tilde{\I}(X(t)) = r^{\top}X(t) = \tilde{\I}(X(t_0))
\end{equation}
for an $r\in \mathbb{R}^{d}$ and the ability of the numerical schemes to preserve it as well.

As we did for quadratic invariants, we calculate $\dmath \tilde{\I}(X(t))$,
\begin{equation*}
  \dmath \tilde{\I}(X(t)) = \sum_{m=0}^{M} \left(r^{\top} A_m X(t) + r^{\top}g_m(X(t))\right)\circ\dW_m(t).
\end{equation*}
Thus for \cref{equ:sde} to preserve linear invariants we need the following two assumptions to be fulfilled:
\begin{assumption}\label{ass:Nullspace}
  The matrices $A_m$, $m=0,\dots,M$, satisfy $r\in \text{\it Null}(A_m^{\top})$.
\end{assumption}
\begin{assumption}
    For all $X\in\mathbb{R}^d$ and $m=0,\dots,M$ we have $r^{\top} g_m(X)=0$.
\end{assumption}
We now note that by \cref{ass:Nullspace} we have that
\begin{equation*}
	r^{\top} e^{L^n(t)} = r^{\top}.
\end{equation*}
Using this property, \cref{lem:TransfPresInv,lem:InvTransSDE,thm:NumQuadInv} extend directly to linear invariants.

\section{Numerical examples}\label{sec:numEx}
In this section four different stochastic Lawson schemes are tested numerically and compared with
the standard midpoint rule: The trapezoidal DSL (TDSL), the trapezoidal FSL (TFSL),
the midpoint DSL (MDSL) and the midpoint FSL (MFSL).
All schemes are of {weak order 1, and of }{strong} order 1 for SDEs with {commutative} noise, otherwise of {strong} order 0.5.
The methods are tested on two oscillatory problems preserving
quadratic invariants. The error vs.\ step size for different degrees of
oscillatory behaviour is measured, as well as the ability of the schemes to preserve the invariants.
Finally, the methods are tested on a stochastic  Fermi-Pasta-Ulam-Tsingou (FPUT) problem. This is highly
oscillatory, but does not preserve quadratic invariants exactly. The main findings of the experiments
are:
\begin{itemize}
  \item The SL schemes resolve the oscillatory behaviour better than the midpoint rule. For small
    linear diffusion terms, there are no significant differences between the DSL and the FSL schemes.
    For high noise problems, the FSL schemes are superior.
    \item The midpoint SL schemes preserve quadratic invariants.
    \item The trapezoidal FSL scheme almost preserves quadratic invariants for linear diffusions, but not for non-linear diffusions.
    \item {The overhead caused by calculating the exponentials is insignificant, even for the
      FSL-schemes.}
\end{itemize}
In all cases, the underlying non-linear algebraic equations of the implicit methods are solved by Newton's method with tolerance $10^{-12}$.
\subsection{Stochastic rigid body problem}\label{sec:StochasticRigidBody}
The first example is the classical rigid body problem perturbed with linear
skew-symmetric drift and diffusion terms,  see e.g. \citet{abdulle12hwo,anmarkrud2017ocf,cohen12otn}. The
SDE is given by

\begin{equation}\label{equ:rigidbody}
\begin{multlined}
    \dX = \omega \begin{pmatrix} 0 & 1 & 0\\ -1 & 0 &0 \\ 0&0&0 \end{pmatrix}X \dt + \begin{pmatrix}
        0 & X_3 / I_3 & -X_2 / I_2 \\
        -X_3 / I_3 & 0 & X_1 / I_1 \\
        X_2 / I_2 & -X_1 / I_1 & 0
    \end{pmatrix}
    X \dt\ifnotarxiv\else\\\fi + \sigma \begin{pmatrix} 0 & 1 & 0\\ -1 & 0 &0 \\ 0&0&0
    \end{pmatrix}X \circ \dW,
    \end{multlined}
\end{equation}
with, following the above references,
\begin{equation}
	I_1 = 2, \quad I_2=1, \quad I_3=2/3, \quad t_0=0 \quad \mbox{and} \quad X(0) = (\cos(1.1), 0, \sin(1.1))^{\top}.
\end{equation}
Different values of the  parameters $\omega$ and $\sigma$ are chosen to investigate the schemes'
response on problems with oscillatory drift and/or diffusion terms.

First, we want to see how well the methods respond to different degrees of oscillatory behaviour.
The following two experiments are carried out:
\begin{enumerate}[(a)]
  \item \label{itema} The oscillatory drift and diffusion terms contribute equally much to
    the solution, thus there can be a significant amount of noise {in} the system. The values used in the experiments
    are $\omega=\sigma\in\{1,5,10\}$.
  \item \label{itemb} The noise term is small {($\sigma=0.3$)}, but the drift term causes no {oscillations} {($\omega=0$)}, average {($\omega=10$)} {or} very high
    oscillations {($\omega=100$)}. When $\omega=0$, the DSL schemes coincide
    with their underlying schemes. The choice of $\sigma$ corresponds to the one used by \citet{cohen12otn}.
\end{enumerate}
The integration interval is $[0,1]$ and the problem is solved with step sizes
$\log_2(h)\in\{-11,\ldots,-3\}$, and the {strong} error is calculated at the end of the interval.
The reference solution is computed by the MFSL scheme, using
$h_{\text{ref}}=2^{-17}$. In each case, the results are based on 1000 independent simulations of the
underlying Wiener process. The results are given in \cref{fig:RigidBodyGlobalError}.
The $95\%$-confidence intervals are in all cases found to be less than $12\%$
of the corresponding error value.

\pgfplotstableread[col
sep=comma]{Data/RigidBody/StochasticRigidBody_OuterIter=1_Miter=40_Mbatch=25_Nh=17_TDSL_TFSL_M_MDSL_MFSL_tend=1_Strong_Param=0_0.3Endpoint.csv}\ConvRigidBodyF
\pgfplotstableread[col sep=comma]{Data/RigidBody/StochasticRigidBody_OuterIter=1_Miter=40_Mbatch=25_Nh=17_TDSL_TFSL_M_MDSL_MFSL_tend=1_Strong_Param=10_0.3Endpoint.csv}\ConvRigidBodyS
\pgfplotstableread[col sep=comma]{Data/RigidBody/StochasticRigidBody_OuterIter=1_Miter=40_Mbatch=25_Nh=17_TDSL_TFSL_M_MDSL_MFSL_tend=1_Strong_Param=100_0.3Endpoint.csv}\ConvRigidBodyT

\pgfplotstableread[col sep=comma]{Data/RigidBody/StochasticRigidBody_OuterIter=1_Miter=40_Mbatch=25_Nh=17_TDSL_TFSL_M_MDSL_MFSL_tend=1_Strong_Param=1_1Endpoint.csv}\ConvRigidBodyFF
\pgfplotstableread[col sep=comma]{Data/RigidBody/StochasticRigidBody_OuterIter=1_Miter=40_Mbatch=25_Nh=17_TDSL_TFSL_M_MDSL_MFSL_tend=1_Strong_Param=5_5Endpoint.csv}\ConvRigidBodySS
\pgfplotstableread[col sep=comma]{Data/RigidBody/StochasticRigidBody_OuterIter=1_Miter=40_Mbatch=25_Nh=17_TDSL_TFSL_M_MDSL_MFSL_tend=1_Strong_Param=10_10Endpoint.csv}\ConvRigidBodyTT

\begin{figure}[ht!]
\centering
\begin{tikzpicture}
\begin{customlegend}[legend columns=5,legend style={align=left,draw=none,column sep=2ex},legend entries={TDSL, TFSL, MDSL, MFSL, Midpoint}]
        \addlegendimage{TDSLstyle}
        \addlegendimage{TFSLstyle}
        \addlegendimage{MDSLstyle}
        \addlegendimage{MFSLstyle}
        \addlegendimage{Mstyle}
\end{customlegend}
\end{tikzpicture}
\centering
\begin{tikzpicture}
\begin{customlegend}[legend columns=1,legend style={align=left,draw=none,column sep=2ex},legend entries={{reference line with slope 1}}]
        \addlegendimage{dotted}
\end{customlegend}
\end{tikzpicture}

\begin{subfigure}[b]{0.5\linewidth}
\centering\captionsetup{width=.8\linewidth}%
\begin{tikzpicture}
\begin{axis}[%
    name=plot1,
    width=0.85\textwidth,
    height=0.15\textheight,
    axis x line=bottom,
    ymode = log,
    xmode = log,
    log basis y={2},
    log basis x={2},
    ytick ={0.5,0.0625,0.0078125,0.0009765625,0.0001220703125,0.00001525878906},
    axis y line=left,
    title={$\omega=\sigma=1$},
    xlabel={$h$},
    ylabel={$E\|X(t_N)-Y_{N}\|$}]

    \addplot[TDSLstyle] table[%
    x={h},
    y={e1_TDSL},
    ]{\ConvRigidBodyFF};

    \addplot[TFSLstyle] table[%
    x={h},
    y={e1_TFSL},
    ]{\ConvRigidBodyFF};

    \addplot[MDSLstyle] table[%
    x={h},
    y={e1_MDSL},
    ]{\ConvRigidBodyFF};

    \addplot[MFSLstyle] table[%
    x={h},
    y={e1_MFSL},
    ]{\ConvRigidBodyFF};

    \addplot[Mstyle] table[%
    x={h},
    y={e1_M},
    ]{\ConvRigidBodyFF};

    \addplot [domain=0.00048828:0.1250,dotted] {2*x};

\end{axis}

\begin{axis}[%
    name=plot2,
    at=(plot1.below south west), anchor=above north west,
    width=0.85\textwidth,
    height=0.15\textheight,
    axis x line=bottom,
    ymode = log,
    xmode = log,
    log basis y={2},
    log basis x={2},
    ytick ={0.5,0.0625,0.0078125,0.0009765625,0.0001220703125,0.00001525878906},
    axis y line=left,
    title={$\omega=\sigma=5$},
    xlabel={$h$},
    ylabel={$E\|X(t_N)-Y_{N}\|$}]

    \addplot[TDSLstyle] table[%
    x={h},
    y={e1_TDSL},
    ]{\ConvRigidBodySS};

    \addplot[TFSLstyle] table[%
    x={h},
    y={e1_TFSL},
    ]{\ConvRigidBodySS};

    \addplot[MDSLstyle] table[%
    x={h},
    y={e1_MDSL},
    ]{\ConvRigidBodySS};

    \addplot[MFSLstyle] table[%
    x={h},
    y={e1_MFSL},
    ]{\ConvRigidBodySS};

    \addplot[Mstyle] table[%
    x={h},
    y={e1_M},
    ]{\ConvRigidBodySS};

    \addplot [domain=0.00048828:0.1250,dotted] {2*x};

\end{axis}

\begin{axis}[%
    name=plot3,
    at=(plot2.below south west), anchor=above north west,
    width=0.85\textwidth,
    height=0.15\textheight,
    axis x line=bottom,
    ymode = log,
    xmode = log,
    log basis y={2},
    log basis x={2},
    ytick ={0.5,0.0625,0.0078125,0.0009765625,0.0001220703125,0.00001525878906},
    axis y line=left,
    title={$\omega=\sigma=10$},
    xlabel={$h$},
    ylabel={$E\|X(t_N)-Y_{N}\|$}]

    \addplot[TDSLstyle] table[%
    x={h},
    y={e1_TDSL},
    ]{\ConvRigidBodyTT};

    \addplot[TFSLstyle] table[%
    x={h},
    y={e1_TFSL},
    ]{\ConvRigidBodyTT};

    \addplot[MDSLstyle] table[%
    x={h},
    y={e1_MDSL},
    ]{\ConvRigidBodyTT};

    \addplot[MFSLstyle] table[%
    x={h},
    y={e1_MFSL},
    ]{\ConvRigidBodyTT};

    \addplot[Mstyle] table[%
    x={h},
    y={e1_M},
    ]{\ConvRigidBodyTT};

    \addplot [domain=0.00048828:0.1250,dotted] {2*x};

\end{axis}
\end{tikzpicture}
\caption{\mbox{}}\label{fig:RigidBodyHighHigh}
\end{subfigure}%
\hspace*{\fill}
\begin{subfigure}[b]{0.5\linewidth}
\centering\captionsetup{width=.8\linewidth}%
\begin{tikzpicture}
\begin{axis}[%
    name=plot1,
    width=0.85\textwidth,
    height=0.15\textheight,
    axis x line=bottom,
    ymode = log,
    xmode = log,
    log basis y={2},
    log basis x={2},
    ytick ={0.5,0.0625,0.0078125,0.0009765625,0.0001220703125,0.00001525878906},
    axis y line=left,
    title={$\omega=0$, $\sigma=0.3$},
    xlabel={$h$},
    ylabel={$E\|X(t_N)-Y_{N}\|$}]

    \addplot[TFSLstyle] table[%
    x={h},
    y={e1_TFSL},
    ]{\ConvRigidBodyF};

    \addplot[MFSLstyle] table[%
    x={h},
    y={e1_MFSL},
    ]{\ConvRigidBodyF};

    \addplot[Mstyle] table[%
    x={h},
    y={e1_M},
    ]{\ConvRigidBodyF};

    \addplot [domain=0.00048828:0.1250,dotted] {x/8};

\end{axis}
\begin{axis}[%
    name=plot2,
    at=(plot1.below south west), anchor=above north west,
    width=0.85\textwidth,
    height=0.15\textheight,
    axis x line=bottom,
    ymode = log,
    xmode = log,
    log basis y={2},
    log basis x={2},
    ytick ={0.5,0.0625,0.0078125,0.0009765625,0.0001220703125,0.00001525878906},
    axis y line=left,
    title={$\omega=10$, $\sigma=0.3$},
    xlabel={$h$},
    ylabel={$E\|X(t_N)-Y_{N}\|$}]

    \addplot[TDSLstyle] table[%
    x={h},
    y={e1_TDSL},
    ]{\ConvRigidBodyS};

    \addplot[TFSLstyle] table[%
    x={h},
    y={e1_TFSL},
    ]{\ConvRigidBodyS};

    \addplot[MDSLstyle] table[%
    x={h},
    y={e1_MDSL},
    ]{\ConvRigidBodyS};

    \addplot[MFSLstyle] table[%
    x={h},
    y={e1_MFSL},
    ]{\ConvRigidBodyS};

    \addplot[Mstyle] table[%
    x={h},
    y={e1_M},
    ]{\ConvRigidBodyS};

    \addplot [domain=0.00048828:0.1250,dotted] {x/16};

\end{axis}
\begin{axis}[%
    name=plot3,
    at=(plot2.below south west), anchor=above north west,
    width=0.85\textwidth,
    height=0.15\textheight,
    axis x line=bottom,
    ymode = log,
    xmode = log,
    log basis y={2},
    log basis x={2},
    ytick ={0.5,0.0625,0.0078125,0.0009765625,0.0001220703125,0.00001525878906},
    axis y line=left,
    title={$\omega=100$, $\sigma=0.3$},
    xlabel={$h$},
    ylabel={$E\|X(t_N)-Y_{N}\|$}]

    \addplot[TDSLstyle] table[%
    x={h},
    y={e1_TDSL},
    ]{\ConvRigidBodyT};

    \addplot[TFSLstyle] table[%
    x={h},
    y={e1_TFSL},
    ]{\ConvRigidBodyT};

    \addplot[MDSLstyle] table[%
    x={h},
    y={e1_MDSL},
    ]{\ConvRigidBodyT};

    \addplot[MFSLstyle] table[%
    x={h},
    y={e1_MFSL},
    ]{\ConvRigidBodyT};

    \addplot[Mstyle] table[%
    x={h},
    y={e1_M},
    ]{\ConvRigidBodyT};

    \addplot [domain=0.00048828:0.1250,dotted] {x/8};

\end{axis}

\end{tikzpicture}
\caption{\mbox{}}\label{fig:RigidBodyHighLow}
\end{subfigure}%
\caption{The stochastic rigid body problem: {Strong} error vs.\  {step size} for different values of
$\sigma$ and $\omega$.}\label{fig:RigidBodyGlobalError}
\end{figure}
The results of experiment \eqref{itema} are shown in \cref{fig:RigidBodyHighHigh}. In this experiment,
there is a significant contribution to the solution from the noise term. All schemes demonstrate an acceptable
performance when $\omega=\sigma=1$, the case with slow oscillations. For higher values of $\omega$
and $\sigma$ the FSL schemes are superior. {In the case of $\sigma=\omega=10$, only the
FSL-schemes are able to produce reliable solutions for $h>2^{-7}$. This behaviour}
is as expected, since the FSL schemes solve the oscillatory parts of the problem exactly.

The results of experiment \eqref{itemb} are given in \cref{fig:RigidBodyHighLow}. Here, the
oscillatory noise is quite small, and the oscillations are dominated by the contribution from the
linear drift in all cases but $\omega=0$. For $\omega=0$, we observe no significant differences
between the methods. The dominating part here is the nonlinear drift term,
which is handled similarly by the schemes. For higher values of $\omega$, all the SL schemes
yield similar results, and all are superior to the midpoint rule. In particular the SL schemes are able to resolve the
high oscillations for larger step sizes, up to $h=2^{-6}$ even when $\omega=100$. The midpoint
rule fails to give reliable results for step sizes above approximately $h=2^{-9}$ in this case.

All schemes exhibit, as expected, {strong} order 1 for sufficiently small {step sizes}.
\begin{figure}[ht!]
\centering
\begin{tikzpicture}
\begin{customlegend}[legend columns=5,legend style={align=left,draw=none,column sep=2ex},legend entries={TDSL, TFSL, MDSL, MFSL, Midpoint}]
        \addlegendimage{TDSLstyle}
        \addlegendimage{TFSLstyle}
        \addlegendimage{MDSLstyle}
        \addlegendimage{MFSLstyle}
        \addlegendimage{Mstyle}
\end{customlegend}
\end{tikzpicture}
\begin{subfigure}[b]{0.5\linewidth}
\centering\captionsetup{width=.8\linewidth}%
\begin{tikzpicture}
\begin{axis}[%
    name=plot1,
    width=0.85\textwidth,
    height=0.15\textheight,
    axis x line=bottom,
    ymode = log,
    xmode = log,
    log basis y={2},
    log basis x={2},
    xtick ={0.5,0.0625,0.0078125,0.0009765625,0.0001220703125,0.00001525878906},
    ytick = {0.25,0.03125,0.00390625},
    axis y line=left,
    title={$\omega=\sigma=1$},
    ylabel={Wall-clock time $[s]$},
    xlabel={$E\|X(t_N)-Y_{N}\|$}]

    \addplot[TDSLstyle] table[%
    y={tTDSL},
    x={e1_TDSL},
    ]{\ConvRigidBodyFF};

    \addplot[TFSLstyle] table[%
    y={tTFSL},
    x={e1_TFSL},
    ]{\ConvRigidBodyFF};

    \addplot[MDSLstyle] table[%
    y={tMDSL},
    x={e1_MDSL},
    ]{\ConvRigidBodyFF};

    \addplot[MFSLstyle] table[%
    y={tMFSL},
    x={e1_MFSL},
    ]{\ConvRigidBodyFF};

    \addplot[Mstyle] table[%
    y={tM},
    x={e1_M},
    ]{\ConvRigidBodyFF};

\end{axis}

\begin{axis}[%
    name=plot2,
    at=(plot1.below south west), anchor=above north west,
    width=0.85\textwidth,
    height=0.15\textheight,
    axis x line=bottom,
    ymode = log,
    xmode = log,
    log basis y={2},
    log basis x={2},
    xtick ={0.5,0.0625,0.0078125,0.0009765625,0.0001220703125,0.00001525878906},
    ytick = {0.25,0.03125,0.00390625},
    axis y line=left,
    title={$\omega=\sigma=5$},
    ylabel={Wall-clock time $[s]$},
    xlabel={$E\|X(t_N)-Y_{N}\|$}]

    \addplot[TDSLstyle] table[%
    y={tTDSL},
    x={e1_TDSL},
    ]{\ConvRigidBodySS};

    \addplot[TFSLstyle] table[%
    y={tTFSL},
    x={e1_TFSL},
    ]{\ConvRigidBodySS};

    \addplot[MDSLstyle] table[%
    y={tMDSL},
    x={e1_MDSL},
    ]{\ConvRigidBodySS};

    \addplot[MFSLstyle] table[%
    y={tMFSL},
    x={e1_MFSL},
    ]{\ConvRigidBodySS};

    \addplot[Mstyle] table[%
    y={tM},
    x={e1_M},
    ]{\ConvRigidBodySS};
\end{axis}

\begin{axis}[%
    name=plot3,
    at=(plot2.below south west), anchor=above north west,
    width=0.85\textwidth,
    height=0.15\textheight,
    axis x line=bottom,
    ymode = log,
    xmode = log,
    log basis y={2},
    log basis x={2},
    xtick ={0.5,0.0625,0.0078125,0.0009765625,0.0001220703125,0.00001525878906},
    ytick = {0.25,0.03125,0.00390625},
    axis y line=left,
    title={$\omega=\sigma=10$},
    ylabel={Wall-clock time $[s]$},
    xlabel={$E\|X(t_N)-Y_{N}\|$}]

    \addplot[TDSLstyle] table[%
    y={tTDSL},
    x={e1_TDSL},
    ]{\ConvRigidBodyTT};

    \addplot[TFSLstyle] table[%
    y={tTFSL},
    x={e1_TFSL},
    ]{\ConvRigidBodyTT};

    \addplot[MDSLstyle] table[%
    y={tMDSL},
    x={e1_MDSL},
    ]{\ConvRigidBodyTT};

    \addplot[MFSLstyle] table[%
    y={tMFSL},
    x={e1_MFSL},
    ]{\ConvRigidBodyTT};

    \addplot[Mstyle] table[%
    y={tM},
    x={e1_M},
    ]{\ConvRigidBodyTT};

\end{axis}
\end{tikzpicture}
\caption{\mbox{}}
\end{subfigure}%
\hspace*{\fill}
\begin{subfigure}[b]{0.5\linewidth}
\centering\captionsetup{width=.8\linewidth}%
\begin{tikzpicture}
\begin{axis}[%
    name=plot1,
    width=0.85\textwidth,
    height=0.15\textheight,
    axis x line=bottom,
    ymode = log,
    xmode = log,
    log basis y={2},
    log basis x={2},
    xtick ={0.5,0.0625,0.0078125,0.0009765625,0.0001220703125,0.00001525878906},
    ytick = {0.25,0.03125,0.00390625},
    axis y line=left,
    title={$\omega=0$, $\sigma=0.3$},
    ylabel={Wall-clock time $[s]$},
    xlabel={$E\|X(t_N)-Y_{N}\|$}]

    \addplot[TFSLstyle] table[%
    y={tTFSL},
    x={e1_TFSL},
    ]{\ConvRigidBodyF};

    \addplot[MFSLstyle] table[%
    y={tMFSL},
    x={e1_MFSL},
    ]{\ConvRigidBodyF};

    \addplot[Mstyle] table[%
    y={tM},
    x={e1_M},
    ]{\ConvRigidBodyF};

\end{axis}
\begin{axis}[%
    name=plot2,
    at=(plot1.below south west), anchor=above north west,
    width=0.85\textwidth,
    height=0.15\textheight,
    axis x line=bottom,
    ymode = log,
    xmode = log,
    log basis y={2},
    log basis x={2},
    xtick ={0.5,0.0625,0.0078125,0.0009765625,0.0001220703125,0.00001525878906},
    ytick = {0.25,0.03125,0.00390625},
    axis y line=left,
    title={$\omega=10$, $\sigma=0.3$},
    ylabel={Wall-clock time $[s]$},
    xlabel={$E\|X(t_N)-Y_{N}\|$}]

    \addplot[TDSLstyle] table[%
    y={tTDSL},
    x={e1_TDSL},
    ]{\ConvRigidBodyS};

    \addplot[TFSLstyle] table[%
    y={tTFSL},
    x={e1_TFSL},
    ]{\ConvRigidBodyS};

    \addplot[MDSLstyle] table[%
    y={tMDSL},
    x={e1_MDSL},
    ]{\ConvRigidBodyS};

    \addplot[MFSLstyle] table[%
    y={tMFSL},
    x={e1_MFSL},
    ]{\ConvRigidBodyS};

    \addplot[Mstyle] table[%
    y={tM},
    x={e1_M},
    ]{\ConvRigidBodyS};
\end{axis}
\begin{axis}[%
    name=plot3,
    at=(plot2.below south west), anchor=above north west,
    width=0.85\textwidth,
    height=0.15\textheight,
    axis x line=bottom,
    ymode = log,
    xmode = log,
    log basis y={2},
    log basis x={2},
    xtick ={0.5,0.0625,0.0078125,0.0009765625,0.0001220703125,0.00001525878906},
    ytick = {0.25,0.03125,0.00390625},
    axis y line=left,
    title={$\omega=100$, $\sigma=0.3$},
    ylabel={Wall-clock time $[s]$},
    xlabel={$E\|X(t_N)-Y_{N}\|$}]

    \addplot[TDSLstyle] table[%
    y={tTDSL},
    x={e1_TDSL},
    ]{\ConvRigidBodyT};

    \addplot[TFSLstyle] table[%
    y={tTFSL},
    x={e1_TFSL},
    ]{\ConvRigidBodyT};

    \addplot[MDSLstyle] table[%
    y={tMDSL},
    x={e1_MDSL},
    ]{\ConvRigidBodyT};

    \addplot[MFSLstyle] table[%
    y={tMFSL},
    x={e1_MFSL},
    ]{\ConvRigidBodyT};

    \addplot[Mstyle] table[%
    y={tM},
    x={e1_M},
    ]{\ConvRigidBodyT};

\end{axis}

\end{tikzpicture}
\caption{\mbox{}}
\end{subfigure}%
\caption{{The stochastic rigid body problem: Wall-clock time per batch of 25 paths vs.\ accuracy for different values of
$\sigma$ and $\omega$.}}\label{fig:EfficiencyRigidBody}
\end{figure}
{In \cref{fig:EfficiencyRigidBody} we depict the computational efforts, measured as wall-clock time
per batch of 25 paths vs.\ the strong error averaged over all batches.
Somewhat surprisingly, the FSL schemes are slightly cheaper in terms of computational work per step.
All the methods are implicit, thus the efforts of solving the nonlinear equations dominate. The
computational overhead required for calculating the exponentials for the FSL methods seems in this
example to be outweighted by the fact that there is no diffusion term to calculate in the right hand
side of the equation, as it is absorbed in the exponential. }

\pgfplotstableread[col
sep=comma]{Data/RigidBodyWeak/StochasticRigidBody_OuterIter=12_Miter=20_Mbatch=2500_TDSL_TFSL_M_MDSL_MFSL_tend=1_Weak_Param=0_0.3TrajectoryMax.csv}\ConvRigidBodyF
\pgfplotstableread[col sep=comma]{Data/RigidBodyWeak/StochasticRigidBody_OuterIter=12_Miter=20_Mbatch=2500_TDSL_TFSL_M_MDSL_MFSL_tend=1_Weak_Param=10_0.3TrajectoryMax.csv}\ConvRigidBodyS
\pgfplotstableread[col sep=comma]{Data/RigidBodyWeak/StochasticRigidBody_OuterIter=12_Miter=20_Mbatch=2500_TDSL_TFSL_M_MDSL_MFSL_tend=1_Weak_Param=100_0.3TrajectoryMax.csv}\ConvRigidBodyT

\pgfplotstableread[col sep=comma]{Data/RigidBodyWeak/StochasticRigidBody_OuterIter=12_Miter=20_Mbatch=2500_TDSL_TFSL_M_MDSL_MFSL_tend=1_Weak_Param=1_1TrajectoryMax.csv}\ConvRigidBodyFF
\pgfplotstableread[col sep=comma]{Data/RigidBodyWeak/StochasticRigidBody_OuterIter=12_Miter=20_Mbatch=2500_TDSL_TFSL_M_MDSL_MFSL_tend=1_Weak_Param=5_5TrajectoryMax.csv}\ConvRigidBodySS
\pgfplotstableread[col sep=comma]{Data/RigidBodyWeak/StochasticRigidBody_OuterIter=12_Miter=20_Mbatch=2500_TDSL_TFSL_M_MDSL_MFSL_tend=1_Weak_Param=10_10TrajectoryMax.csv}\ConvRigidBodyTT

\begin{figure}[ht!]
\centering
\begin{tikzpicture}
\begin{customlegend}[legend columns=5,legend style={align=left,draw=none,column sep=2ex},legend entries={TDSL, TFSL, MDSL, MFSL, Midpoint}]
        \addlegendimage{TDSLstyle}
        \addlegendimage{TFSLstyle}
\end{customlegend}
\end{tikzpicture}
\centering
\begin{tikzpicture}
\begin{customlegend}[legend columns=2,legend style={align=left,draw=none,column sep=2ex},legend entries={reference line with slope 1,reference line with slope 2}]
        \addlegendimage{dotted}
        \addlegendimage{dashdotted}
\end{customlegend}
\end{tikzpicture}

\begin{subfigure}[b]{0.5\linewidth}
\centering\captionsetup{width=.8\linewidth}%
\begin{tikzpicture}
\begin{axis}[%
    name=plot1,
    width=0.85\textwidth,
    height=0.15\textheight,
    axis x line=bottom,
    ymode = log,
    xmode = log,
    log basis y={2},
    log basis x={2},
    axis y line=left,
    title={$\omega=\sigma=1$},
    xlabel={$h$},
    ylabel={$|E(I(Y_{N})-1)|$}]

    \addplot[TDSLstyle] table[%
    x={h},
    y={e1_TDSL},
    ]{\ConvRigidBodyFF};

    \addplot[TFSLstyle] table[%
    x={h},
    y={e1_TFSL},
    ]{\ConvRigidBodyFF};

    \addplot [domain=0.00048828:0.1250,dotted] {x/256};
    \addplot [domain=0.00048828:0.1250,dashdotted] {x*x/256};

\end{axis}

\begin{axis}[%
    name=plot2,
    at=(plot1.below south west), anchor=above north west,
    width=0.85\textwidth,
    height=0.15\textheight,
    axis x line=bottom,
    ymode = log,
    xmode = log,
    log basis y={2},
    log basis x={2},
    axis y line=left,
    title={$\omega=\sigma=5$},
    xlabel={$h$},
    ylabel={$|E(I(Y_{N})-1)|$}]

    \addplot[TDSLstyle] table[%
    x={h},
    y={e1_TDSL},
    ]{\ConvRigidBodySS};

    \addplot[TFSLstyle] table[%
    x={h},
    y={e1_TFSL},
    ]{\ConvRigidBodySS};

    \addplot [domain=0.00048828:0.1250,dotted] {x/256};
    \addplot [domain=0.00048828:0.1250,dashdotted] {x*x/256};

\end{axis}

\begin{axis}[%
    name=plot3,
    at=(plot2.below south west), anchor=above north west,
    width=0.85\textwidth,
    height=0.15\textheight,
    axis x line=bottom,
    ymode = log,
    xmode = log,
    log basis y={2},
    log basis x={2},
    axis y line=left,
    title={$\omega=\sigma=10$},
    xlabel={$h$},
    ylabel={$|E(I(Y_{N})-1)|$}]

    \addplot[TDSLstyle] table[%
    x={h},
    y={e1_TDSL},
    ]{\ConvRigidBodyTT};

    \addplot[TFSLstyle] table[%
    x={h},
    y={e1_TFSL},
    ]{\ConvRigidBodyTT};

    \addplot [domain=0.00048828:0.1250,dotted] {x/256};
    \addplot [domain=0.00048828:0.1250,dashdotted] {x*x/256};

\end{axis}
\end{tikzpicture}
\caption{\mbox{}}
\end{subfigure}%
\hspace*{\fill}
\begin{subfigure}[b]{0.5\linewidth}
\centering\captionsetup{width=.8\linewidth}%
\begin{tikzpicture}
\begin{axis}[%
    name=plot1,
    width=0.85\textwidth,
    height=0.15\textheight,
    axis x line=bottom,
    ymode = log,
    xmode = log,
    log basis y={2},
    log basis x={2},
    axis y line=left,
    title={$\omega=0$, $\sigma=0.3$},
    xlabel={$h$},
    ylabel={$|E(I(Y_{N})-1)|$}]

    \addplot[TFSLstyle] table[%
    x={h},
    y={e1_TFSL},
    ]{\ConvRigidBodyF};

    \addplot [domain=0.00048828:0.1250,dashdotted] {x*x/256};

\end{axis}
\begin{axis}[%
    name=plot2,
    at=(plot1.below south west), anchor=above north west,
    width=0.85\textwidth,
    height=0.15\textheight,
    axis x line=bottom,
    ymode = log,
    xmode = log,
    log basis y={2},
    log basis x={2},
    axis y line=left,
    title={$\omega=10$, $\sigma=0.3$},
    xlabel={$h$},
    ylabel={$|E(I(Y_{N})-1)|$}]

    \addplot[TDSLstyle] table[%
    x={h},
    y={e1_TDSL},
    ]{\ConvRigidBodyS};

    \addplot[TFSLstyle] table[%
    x={h},
    y={e1_TFSL},
    ]{\ConvRigidBodyS};

    \addplot [domain=0.00048828:0.1250,dotted] {x/256};
    \addplot [domain=0.00048828:0.1250,dashdotted] {x*x/256};

\end{axis}
\begin{axis}[%
    name=plot3,
    at=(plot2.below south west), anchor=above north west,
    width=0.85\textwidth,
    height=0.15\textheight,
    axis x line=bottom,
    ymode = log,
    xmode = log,
    log basis y={2},
    log basis x={2},
    axis y line=left,
    title={$\omega=100$, $\sigma=0.3$},
    xlabel={$h$},
    ylabel={$|E(I(Y_{N})-1)|$}]

    \addplot[TDSLstyle] table[%
    x={h},
    y={e1_TDSL},
    ]{\ConvRigidBodyT};

    \addplot[TFSLstyle] table[%
    x={h},
    y={e1_TFSL},
    ]{\ConvRigidBodyT};

    \addplot [domain=0.00048828:0.1250,dotted] {x/256};
    \addplot [domain=0.00048828:0.1250,dashdotted] {x*x/256};

\end{axis}

\end{tikzpicture}
\caption{\mbox{}}
\end{subfigure}%
\caption{{The stochastic rigid body problem: Weak error of $I$ vs.\  step size for different values of
$\sigma$ and $\omega$.}}\label{fig:RigidBodyGlobalWeakError}
\end{figure}

We will next see how well the methods preserve invariants, also over long time integration.
In the case of the stochastic rigid body problem \cref{equ:rigidbody},
\cref{ass:ABskew,ass:SkewSym} are satisfied, and by \cref{lem:InvSDE} the exact solution of \cref{equ:rigidbody} preserves the invariant
\begin{equation}\label{eq:Invrigitbody}
    \mathcal{I}(X(t)) = X(t)^{\top}  X(t)
\end{equation}
and thus, when using the given initial values, stays  on the unit sphere.
The matrices $A_0$ and $A_1$ commute, so the requirements of \cref{cor:QInvMidpoint} are satisfied
and the MDSL and MFSL schemes should preserve the invariant.
As $g_1(x)=0$ the requirements for
\cref{cor:TrapzruleQinvariant} are satisfied, and the TFSL scheme {nearly} preserves the
invariant, and the weak order with respect to $I$ is 2. This is clearly demonstrated in
\cref{fig:RigidBodyGlobalWeakError}, the weak order is one for the TDSL scheme, and two for the TFSL
scheme. For the small noise problem, the error from the drift term
  dominates, and the order, in particular for larger values of $h$ is then close to 2. The midpoint schemes {preserve} the invariant exact, and {are} not included in the figure.

Lastly, as $g_0(x)\neq 0$, Theorem 3.1 in
\citet{cohen12otn} does not apply, and we expect the TDSL scheme to drift off. {To confirm this, two extreme cases from the experiments above have been chosen, that is}
$\omega=\sigma=10$, and $\omega=100$ and $\sigma=0.3$. The SDE is solved on the interval
$[0,100]$, using the {step size} $h=25/2^{10}$. \Cref{fig:InvRigidBody} shows the value of the
invariant $\I(Y_n)$. For visibility, only each 128th step is plotted. In
\cref{fig:InvRigidBodyzoom}, we zoom in onto the interval $[60,70]$, for better visibility
neglecting the TFSL method and for the midpoint methods plotting only every 6th value.  Notice that the drift of the TDSL schemes is
severe in the experiments with much noise, while in the case of small noise, it is reasonably
small.

\begin{figure}[ht!]
\centering
\pgfplotsset{
    legend image with TFSL MDSL MFSL Midpoint overlaid/.style={
        legend image code/.code={%
            \LegendImage{/pgfplots/TFSLstyle}
            \LegendImage{/pgfplots/MFSLstyle}
            \LegendImage{/pgfplots/MDSLstyle}
             \LegendImage{/pgfplots/Mstyle}
        }
    },
}

\begin{tikzpicture}
\begin{customlegend}[legend columns=5,legend style={align=left,draw=none,column sep=2ex},legend entries={TDSL\\{TFSL, MDSL, MFSL, Midpoint} {(not distinguishable on these plots)}\\}]
        \addlegendimage{TDSLstyle}
        \addlegendimage{legend image with TFSL MDSL MFSL Midpoint overlaid}
\end{customlegend}
\end{tikzpicture}
\begin{tikzpicture}
\pgfplotstableread[col sep=comma]{Data/RigidBody/StochasticRigidBody_OuterIter=1_Miter=1_Mbatch=1_TDSL_TFSL_M_MDSL_MFSL_tend=100Param=10_10_meanvalues_grid_1_coarse=128.csv}\RigidTenTen
\begin{axis}[%
    name=plot1,
    width=0.9\linewidth,
    height=0.135\textheight,
    axis x line=bottom,
    axis y line=left,
    title={$\omega= \sigma=10$},
    xlabel={$t_n$},
    ylabel={$\I(Y_n)$},
    legend pos=north west,
    scaled y ticks = false,
    log ticks with fixed point,
    yticklabel style={/pgf/number format/.cd,fixed,precision=2, fixed zerofill},]

    \addplot[TDSLstyle] table[x={t}, y expr={(\thisrow{I1_TDSL})}]{\RigidTenTen};

    \addplot[TFSLstyle] table[x={t}, y expr={(\thisrow{I1_TFSL})}]{\RigidTenTen};

   \addplot[MDSLstyle] table[x={t}, y expr={(\thisrow{I1_MDSL})}]{\RigidTenTen};

    \addplot[MFSLstyle] table[x={t}, y expr={(\thisrow{I1_MFSL})}]{\RigidTenTen};

    \addplot[Mstyle] table[x={t}, y expr={(\thisrow{I1_M})}]{\RigidTenTen};
\end{axis}
\pgfplotstableread[col sep=comma]{Data/RigidBody/StochasticRigidBody_OuterIter=1_Miter=1_Mbatch=1_TDSL_TFSL_M_MDSL_MFSL_tend=100Param=100_0.3_meanvalues_grid_1_coarse=128.csv}\RigidHundredZeroThree
\begin{axis}[%
    name=plot2,
    at=(plot1.below south west), anchor=above north west,
    width=0.9\linewidth,
    height=0.135\textheight,
    axis x line=bottom,
    axis y line=left,
    title={$\omega=100$, $\sigma=0.3$},
    xlabel={$t_n$},
    ylabel={$\I(Y_n)$},
    legend pos=north west,
    scaled y ticks = false,
    log ticks with fixed point,
    yticklabel style={/pgf/number format/.cd,fixed,precision=4, fixed zerofill},]

    \addplot[TDSLstyle] table[x={t}, y expr={(\thisrow{I1_TDSL})}]{\RigidHundredZeroThree};

    \addplot[TFSLstyle] table[x={t}, y expr={(\thisrow{I1_TFSL})}]{\RigidHundredZeroThree};

   \addplot[MDSLstyle] table[x={t}, y expr={(\thisrow{I1_MDSL})}]{\RigidHundredZeroThree};

    \addplot[MFSLstyle] table[x={t}, y expr={(\thisrow{I1_MFSL})}]{\RigidHundredZeroThree};

    \addplot[Mstyle] table[x={t}, y expr={(\thisrow{I1_M})}]{\RigidHundredZeroThree};
\end{axis}

    \end{tikzpicture}
    \caption{The stochastic rigid body problem: The invariant $\I(Y_n)$.} \label{fig:InvRigidBody}
\end{figure}

\begin{figure}[ht!]
\centering
\begin{tikzpicture}

\begin{customlegend}[legend columns=5,legend style={align=left,draw=none,column sep=2ex},legend entries={TFSL\\{MDSL, MFSL, Midpoint} {(not distinguishable on these plots)}\\}]
        \addlegendimage{TFSLstyle,mark=none}
        \addlegendimage{legend image with MDSL MFSL Midpoint overlaid}
\end{customlegend}
\end{tikzpicture}
\begin{tikzpicture}
\pgfplotstableread[col sep=comma]{Data/RigidBody/StochasticRigidBody_OuterIter=1_Miter=1_Mbatch=1_TDSL_TFSL_M_MDSL_MFSL_tend=100Param=10_10_meanvalues_grid_1_60_70.csv}\RigidTenTen
\begin{axis}[%
    name=plot1,
    width=0.9\linewidth,
    height=0.135\textheight,
    axis x line=bottom,
    axis y line=left,
    ymax=1.00002,
    title={$\omega= \sigma=10$},
    xlabel={$t_n$},
    ylabel={$\I(Y_n)$},
    legend pos=north west,
    scaled y ticks = false,
    log ticks with fixed point,
    yticklabel style={/pgf/number format/.cd,fixed,precision=6, fixed zerofill},]

    \addplot[TFSLstyle,mark=none] table[x={t}, y expr={(\thisrow{I1_TFSL})}]{\RigidTenTen};

   \addplot[MDSLstyle,each nth point={6}] table[x={t}, y expr={(\thisrow{I1_MDSL})}]{\RigidTenTen};

    \addplot[MFSLstyle,each nth point={6}] table[x={t}, y expr={(\thisrow{I1_MFSL})}]{\RigidTenTen};

    \addplot[Mstyle,each nth point={6}] table[x={t}, y expr={(\thisrow{I1_M})}]{\RigidTenTen};
\end{axis}
\pgfplotstableread[col sep=comma]{Data/RigidBody/StochasticRigidBody_OuterIter=1_Miter=1_Mbatch=1_TDSL_TFSL_M_MDSL_MFSL_tend=100Param=100_0.3_meanvalues_grid_1_60_70.csv}\RigidHundredZeroThree
\begin{axis}[%
    name=plot2,
    at=(plot1.below south west), anchor=above north west,
    width=0.9\linewidth,
    height=0.135\textheight,
    axis x line=bottom,
    axis y line=left,
    title={$\omega=100$, $\sigma=0.3$},
    xlabel={$t_n$},
    ylabel={$\I(Y_n)$},
    legend pos=north west,
    scaled y ticks = false,
    log ticks with fixed point,
    yticklabel style={/pgf/number format/.cd,fixed,precision=6, fixed zerofill},]

    \addplot[TFSLstyle,mark=none] table[x={t}, y expr={(\thisrow{I1_TFSL})}]{\RigidHundredZeroThree};

   \addplot[MDSLstyle,each nth point={6}] table[x={t}, y expr={(\thisrow{I1_MDSL})}]{\RigidHundredZeroThree};

    \addplot[MFSLstyle,each nth point={6}] table[x={t}, y expr={(\thisrow{I1_MFSL})}]{\RigidHundredZeroThree};

    \addplot[Mstyle,each nth point={6}] table[x={t}, y expr={(\thisrow{I1_M})}]{\RigidHundredZeroThree};
\end{axis}

    \end{tikzpicture}
    \caption{The stochastic rigid body problem: The invariant $\I(Y_n)$.} \label{fig:InvRigidBodyzoom}
\end{figure}

To this aim, the stochastic rigid problem is solved by the
TFSL scheme over the interval $[0,100]$ for different step sizes, with 1000 independent
simulations. \cref{fig:AlmostPreserve} shows the maximum average deviation over the interval,
clearly demonstrating the second order of the deviation predicted by
\cref{cor:TrapzruleQinvariant}. The parameters used in this experiment are $\omega=10$ and
$\sigma=0.3$.

\begin{figure}[ht!]
\begin{tikzpicture}
\pgfplotstableread[col sep=comma]{Data/RigidBody/StochasticRigidBody_OuterIter=1_Miter=40_Mbatch=25_TFSL_tend=100_Strong_Param=10_0.3TrajectoryMax.csv}\TFSLNearPreserve
\begin{axis}[%
    name=plot1,
    width=0.9\linewidth,
    height=0.15\textheight,
    axis x line=bottom,
    ymode = log,
    xmode = log,
    log basis y={2},
    log basis x={2},
    xtick = { 0.125,0.0625,0.03125,0.015625,0.0078125,0.00390625,0.001953125,0.0009765625,0.00048828125},
    axis y line=left,
    legend pos=north west,
    xlabel={$h$},
    ylabel={$\max_{n}\E|\I(Y_n)-1|$}]

    \addplot[TFSLstyle,only marks] table[x={h},y={e1_TFSL}]{\TFSLNearPreserve};
	\addplot[TFSLstyle,mark=none,forget plot] table[x={h}, y={create col/linear regression={y=e1_TFSL}}]{\TFSLNearPreserve};
    	\xdef\slopeTFSL{\pgfplotstableregressiona}
		\addlegendentry{TFSL, $p=\pgfmathprintnumber{\slopeTFSL}$}
\end{axis}
\end{tikzpicture}
\caption{The stochastic rigid body problem with $\omega=10$ and $\sigma=0.3$: Maximum average deviation on the interval [0,100] vs.\  step size when solved by the TFSL method. {$p$ denotes the numerically determined order of convergence.}} \label{fig:AlmostPreserve}
\end{figure}

\subsection{Non-linear Kubo oscillator}
We next study the methods applied to a problem with multiple noise terms. To this aim,
the non-linear Kubo oscillator, already introduced in \cref{sec:intro}, is chosen. The SDE is
\begin{equation}\label{equ:NonLinKuboa}
\dX(t) = \sum_{m=0}^M \left[ \omega_m \begin{pmatrix}
0 & -1 \\ 1 & 0
\end{pmatrix} X(t) + \begin{pmatrix}
0 & -U_m(X(t)) \\ U_m(X(t)) & 0
\end{pmatrix} X(t) \right] \circ \dW_m(t)
\end{equation}
with $U_m:\mathbb{R}^2 \to \mathbb{R}$. In the current experiments, $M=2$ and the nonlinear terms
are
\begin{equation*}
	U_0(X) = \frac{1}{5}(X_1 + X_2)^5, ~ U_1(X)=0, ~ U_2(X)=\frac{1}{3}(X_1+X_2)^3,~ \omega_0=
	\omega, ~ \omega_1=\sigma, ~ \omega_2=0,
\end{equation*}
thus there are two independent diffusion terms, one linear and one non-linear.
These are scaled to ensure that the highly oscillatory parts come from the linear terms. As in the
rigid body case, $\omega$ and $\sigma$ vary to investigate how the schemes respond to highly
oscillatory parts in the drift and / or diffusion terms.
The experiment setup is as for the stochastic rigid problem:
\begin{enumerate}[(a)]
\item Equal contribution from drift and diffusion: The parameters used in the experiments are  $\omega=\sigma\in\{1,5,10\}$,
\item Small linear diffusion: The parameters are  $\sigma=0.3$, and $\omega\in\{1,10,50\}$.
\end{enumerate}

The integration interval is [0,1], {$X(0)=(1,0)^\top$}, and the SDE is solved with step sizes
$\log_2(h)\in\{-15,\ldots,-8$\}, and with 1000 independent simulations. The reference solutions are
computed by the  MFSL scheme using $h_{\text{ref}}=2^{-21}$. The {strong} error at the end
point is shown in \cref{fig:KuboConv}. The $95\%$-confidence intervals have been calculated and span
in all cases less than $15\%$ of the corresponding error values.

\pgfplotstableread[col sep=comma]{Data/Kubo/Non-linearKubooscillator_OuterIter=1_Miter=40_Mbatch=25_Nh=21_TDSL_TFSL_M_MDSL_MFSL_tend=1_Strong_Param=1_1Endpoint.csv}\convergenceOsc
\pgfplotstableread[col sep=comma]{Data/Kubo/Non-linearKubooscillator_OuterIter=1_Miter=40_Mbatch=25_Nh=21_TDSL_TFSL_M_MDSL_MFSL_tend=1_Strong_Param=5_5Endpoint.csv}\convergenceMedOsc
\pgfplotstableread[col sep=comma]{Data/Kubo/Non-linearKubooscillator_OuterIter=1_Miter=40_Mbatch=25_Nh=21_TDSL_TFSL_M_MDSL_MFSL_tend=1_Strong_Param=10_10Endpoint.csv}\convergence

\pgfplotstableread[col sep=comma]{Data/Kubo/Non-linearKubooscillator_OuterIter=1_Miter=40_Mbatch=25_Nh=21_TDSL_TFSL_M_MDSL_MFSL_tend=1_Strong_Param=1_0.3Endpoint.csv}\convergenceOscNew
\pgfplotstableread[col sep=comma]{Data/Kubo/Non-linearKubooscillator_OuterIter=1_Miter=40_Mbatch=25_Nh=21_TDSL_TFSL_M_MDSL_MFSL_tend=1_Strong_Param=10_0.3Endpoint.csv}\convergenceMedOscNew
\pgfplotstableread[col sep=comma]{Data/Kubo/Non-linearKubooscillator_OuterIter=1_Miter=40_Mbatch=25_Nh=21_TDSL_TFSL_M_MDSL_MFSL_tend=1_Strong_Param=50_0.3Endpoint.csv}\convergenceNew
\begin{figure}[ht!]
\centering
\begin{tikzpicture}
\centering
\begin{customlegend}[legend columns=5,legend style={align=left,draw=none,column sep=2ex},legend entries={TDSL, TFSL, MDSL, MFSL, Midpoint}]
        \addlegendimage{TDSLstyle}
        \addlegendimage{TFSLstyle}
        \addlegendimage{MDSLstyle}
        \addlegendimage{MFSLstyle}
        \addlegendimage{Mstyle}
\end{customlegend}
\end{tikzpicture}
\centering
\begin{tikzpicture}
\begin{customlegend}[legend columns=2,legend style={align=left,draw=none,column sep=2ex},legend entries={{reference line with slope 0.5},{reference line with slope 1}}]
        \addlegendimage{dashed}
        \addlegendimage{dotted}
\end{customlegend}
\end{tikzpicture}

\begin{subfigure}[b]{0.5\linewidth}
\centering\captionsetup{width=.8\linewidth}%
\begin{tikzpicture}
\begin{axis}[%
    name=plot1,
    width=0.95\linewidth,
    height=0.15\textheight,
    axis x line=bottom,
    ymode = log,
    xmode = log,
    log basis y={2},
    log basis x={2},
    xtick ={data},
    axis y line=left,
    title={$\omega=\sigma=1$},
    xlabel={$h$},
    ylabel={$E\|X(t_N)-Y_{N}\|$}]

    \addplot[TDSLstyle] table[%
    x={h},
    y={e1_TDSL},
    ]{\convergenceOsc};

    \addplot[TFSLstyle] table[%
    x={h},
    y={e1_TFSL},
    ]{\convergenceOsc};

    \addplot[MDSLstyle] table[%
    x={h},
    y={e1_MDSL},
    ]{\convergenceOsc};

    \addplot[MFSLstyle] table[%
    x={h},
    y={e1_MFSL},
    ]{\convergenceOsc};

    \addplot[Mstyle] table[%
    x={h},
    y={e1_M},
    ]{\convergenceOsc};

    \addplot [domain=0.0000305:0.0039,dashed] {sqrt(x/2)};

\end{axis}

\begin{axis}[%
    name=plot2,
    at=(plot1.below south west), anchor=above north west,
    width=0.95\linewidth,
    height=0.15\textheight,
    axis x line=bottom,
    ymode = log,
    xmode = log,
    log basis y={2},
    log basis x={2},
    xtick ={data},
    axis y line=left,
    title={$\omega=\sigma=5$},
    xlabel={$h$},
    ylabel={$E\|X(t_N)-Y_{N}\|$}]

    \addplot[TDSLstyle] table[%
    x={h},
    y={e1_TDSL},
    ]{\convergenceMedOsc};

    \addplot[TFSLstyle] table[%
    x={h},
    y={e1_TFSL},
    ]{\convergenceMedOsc};

    \addplot[MDSLstyle] table[%
    x={h},
    y={e1_MDSL},
    ]{\convergenceMedOsc};

    \addplot[MFSLstyle] table[%
    x={h},
    y={e1_MFSL},
    ]{\convergenceMedOsc};

    \addplot[Mstyle] table[%
    x={h},
    y={e1_M},
    ]{\convergenceMedOsc};

    \addplot [domain=0.0000305:0.0039,dashed] {sqrt(x)};
\addplot [domain=0.0000305:0.0039,dotted] {256*x};

\end{axis}

\begin{axis}[%
    name=plot3,
    at=(plot2.below south west), anchor=above north west,
    width=0.95\linewidth,
    height=0.15\textheight,
    axis x line=bottom,
    ymode = log,
    xmode = log,
    log basis y={2},
    log basis x={2},
    xtick ={data},
    axis y line=left,
    title={$\omega=\sigma=10$},
    xlabel={$h$},
    ylabel={$E\|X(t_N)-Y_{N}\|$}]

    \addplot[TDSLstyle] table[%
    x={h},
    y={e1_TDSL},
    ]{\convergence};

    \addplot[TFSLstyle] table[%
    x={h},
    y={e1_TFSL},
    ]{\convergence};

    \addplot[MDSLstyle] table[%
    x={h},
    y={e1_MDSL},
    ]{\convergence};

    \addplot[MFSLstyle] table[%
    x={h},
    y={e1_MFSL},
    ]{\convergence};

    \addplot[Mstyle] table[%
    x={h},
    y={e1_M},
    ]{\convergence};

    \addplot [domain=0.0000305:0.0039,dashed] {2*sqrt(x)};
    \addplot [domain=0.0000305:0.0039,dotted] {1024*x};

\end{axis}
\end{tikzpicture}
\caption{\mbox{}} \label{fig:KuboConvHH}
\end{subfigure}%
\hspace*{\fill}
\begin{subfigure}[b]{0.5\linewidth}
\centering\captionsetup{width=.8\linewidth}%
\begin{tikzpicture}
\begin{axis}[%
    name=plot1,
    width=0.95\linewidth,
    height=0.15\textheight,
    axis x line=bottom,
    ymode = log,
    xmode = log,
    log basis y={2},
    log basis x={2},
    xtick ={data},
    axis y line=left,
    title={$\omega=1$, $\sigma=0.3$},
    xlabel={$h$},
    ylabel={$E\|X(t_N)-Y_{N}\|$}]

    \addplot[TDSLstyle] table[%
    x={h},
    y={e1_TDSL},
    ]{\convergenceOscNew};

    \addplot[TFSLstyle] table[%
    x={h},
    y={e1_TFSL},
    ]{\convergenceOscNew};

    \addplot[MDSLstyle] table[%
    x={h},
    y={e1_MDSL},
    ]{\convergenceOscNew};

    \addplot[MFSLstyle] table[%
    x={h},
    y={e1_MFSL},
    ]{\convergenceOscNew};

    \addplot[Mstyle] table[%
    x={h},
    y={e1_M},
    ]{\convergenceOscNew};

    \addplot [domain=0.0000305:0.0039,dashed] {sqrt(x/32)};

\end{axis}

\begin{axis}[%
    name=plot2,
    at=(plot1.below south west), anchor=above north west,
    width=0.95\linewidth,
    height=0.15\textheight,
    axis x line=bottom,
    ymode = log,
    xmode = log,
    log basis y={2},
    log basis x={2},
    xtick ={data},
    axis y line=left,
    title={$\omega=10$, $\sigma=0.3$},
    xlabel={$h$},
    ylabel={$E\|X(t_N)-Y_{N}\|$}]

    \addplot[TDSLstyle] table[%
    x={h},
    y={e1_TDSL},
    ]{\convergenceMedOscNew};

    \addplot[TFSLstyle] table[%
    x={h},
    y={e1_TFSL},
    ]{\convergenceMedOscNew};

    \addplot[MDSLstyle] table[%
    x={h},
    y={e1_MDSL},
    ]{\convergenceMedOscNew};

    \addplot[MFSLstyle] table[%
    x={h},
    y={e1_MFSL},
    ]{\convergenceMedOscNew};

    \addplot[Mstyle] table[%
    x={h},
    y={e1_M},
    ]{\convergenceMedOscNew};

\addplot [domain=0.0000305:0.0039,dashed] {sqrt(x)/16};
\addplot [domain=0.0000305:0.0039,dotted] {16*x};

\end{axis}

\begin{axis}[%
    name=plot3,
    at=(plot2.below south west), anchor=above north west,
    width=0.95\linewidth,
    height=0.15\textheight,
    axis x line=bottom,
    ymode = log,
    xmode = log,
    log basis y={2},
    log basis x={2},
    xtick ={data},
    axis y line=left,
    title={$\omega=50$, $\sigma=0.3$},
    xlabel={$h$},
    ylabel={$E\|X(t_N)-Y_{N}\|$}]

    \addplot[TDSLstyle] table[%
    x={h},
    y={e1_TDSL},
    ]{\convergenceNew};

    \addplot[TFSLstyle] table[%
    x={h},
    y={e1_TFSL},
    ]{\convergenceNew};

    \addplot[MDSLstyle] table[%
    x={h},
    y={e1_MDSL},
    ]{\convergenceNew};

    \addplot[MFSLstyle] table[%
    x={h},
    y={e1_MFSL},
    ]{\convergenceNew};

    \addplot[Mstyle] table[%
    x={h},
    y={e1_M},
    ]{\convergenceNew};

\addplot [domain=0.0000305:0.0039,dashed] {sqrt(x)/16};
\addplot [domain=0.0000305:0.0039,dotted] {32*x};

\end{axis}

\end{tikzpicture}
\caption{\mbox{}}\label{fig:KuboConvHL}
\end{subfigure}%
\caption{The Kubo oscillator: {Strong} error vs.\ step size for different values of $\sigma$ and
$\omega$. }\label{fig:KuboConv}
\end{figure}

\begin{figure}[ht!]
\centering
\begin{tikzpicture}
\centering
\begin{customlegend}[legend columns=5,legend style={align=left,draw=none,column sep=2ex},legend entries={TDSL, TFSL, MDSL, MFSL, Midpoint}]
        \addlegendimage{TDSLstyle}
        \addlegendimage{TFSLstyle}
        \addlegendimage{MDSLstyle}
        \addlegendimage{MFSLstyle}
        \addlegendimage{Mstyle}
\end{customlegend}
\end{tikzpicture}

\begin{subfigure}[b]{0.5\linewidth}
\centering\captionsetup{width=.8\linewidth}%
\begin{tikzpicture}
\begin{axis}[%
    name=plot1,
    width=0.95\linewidth,
    height=0.15\textheight,
    axis x line=bottom,
    ymode = log,
    xmode = log,
    log basis y={2},
    log basis x={2},
    axis y line=left,
    title={$\omega=\sigma=1$},
    ylabel={Wall-clock time $[s]$},
    xlabel={$E\|X(t_N)-Y_{N}\|$}]

    \addplot[TDSLstyle] table[%
    y={tTDSL},
    x={e1_TDSL},
    ]{\convergenceOsc};

    \addplot[TFSLstyle] table[%
    y={tTFSL},
    x={e1_TFSL},
    ]{\convergenceOsc};

    \addplot[MDSLstyle] table[%
    y={tMDSL},
    x={e1_MDSL},
    ]{\convergenceOsc};

    \addplot[MFSLstyle] table[%
    y={tMFSL},
    x={e1_MFSL},
    ]{\convergenceOsc};

    \addplot[Mstyle] table[%
    y={tM},
    x={e1_M},
    ]{\convergenceOsc};

\end{axis}

\begin{axis}[%
    name=plot2,
    at=(plot1.below south west), anchor=above north west,
    width=0.95\linewidth,
    height=0.15\textheight,
    axis x line=bottom,
    ymode = log,
    xmode = log,
    log basis y={2},
    log basis x={2},
    axis y line=left,
    title={$\omega=\sigma=5$},
    ylabel={Wall-clock time $[s]$},
    xlabel={$E\|X(t_N)-Y_{N}\|$}]

    \addplot[TDSLstyle] table[%
    y={tTDSL},
    x={e1_TDSL},
    ]{\convergenceMedOsc};

    \addplot[TFSLstyle] table[%
    y={tTFSL},
    x={e1_TFSL},
    ]{\convergenceMedOsc};

    \addplot[MDSLstyle] table[%
    y={tMDSL},
    x={e1_MDSL},
    ]{\convergenceMedOsc};

    \addplot[MFSLstyle] table[%
    y={tMFSL},
    x={e1_MFSL},
    ]{\convergenceMedOsc};

    \addplot[Mstyle] table[%
    y={tM},
    x={e1_M},
    ]{\convergenceMedOsc};

\end{axis}

\begin{axis}[%
    name=plot3,
    at=(plot2.below south west), anchor=above north west,
    width=0.95\linewidth,
    height=0.15\textheight,
    axis x line=bottom,
    ymode = log,
    xmode = log,
    log basis y={2},
    log basis x={2},
    axis y line=left,
    title={$\omega=\sigma=10$},
    ylabel={Wall-clock time $[s]$},
    xlabel={$E\|X(t_N)-Y_{N}\|$}]

    \addplot[TDSLstyle] table[%
    y={tTDSL},
    x={e1_TDSL},
    ]{\convergence};

    \addplot[TFSLstyle] table[%
    y={tTFSL},
    x={e1_TFSL},
    ]{\convergence};

    \addplot[MDSLstyle] table[%
    y={tMDSL},
    x={e1_MDSL},
    ]{\convergence};

    \addplot[MFSLstyle] table[%
    y={tMFSL},
    x={e1_MFSL},
    ]{\convergence};

    \addplot[Mstyle] table[%
    y={tM},
    x={e1_M},
    ]{\convergence};

\end{axis}
\end{tikzpicture}
\caption{\mbox{}} \label{fig:KuboConvHHb}
\end{subfigure}%
\hspace*{\fill}
\begin{subfigure}[b]{0.5\linewidth}
\centering\captionsetup{width=.8\linewidth}%
\begin{tikzpicture}
\begin{axis}[%
    name=plot1,
    width=0.95\linewidth,
    height=0.15\textheight,
    axis x line=bottom,
    ymode = log,
    xmode = log,
    log basis y={2},
    log basis x={2},
    axis y line=left,
    title={$\omega=1$, $\sigma=0.3$},
    ylabel={Wall-clock time $[s]$},
    xlabel={$E\|X(t_N)-Y_{N}\|$}]

    \addplot[TDSLstyle] table[%
    y={tTDSL},
    x={e1_TDSL},
    ]{\convergenceOscNew};

    \addplot[TFSLstyle] table[%
    y={tTFSL},
    x={e1_TFSL},
    ]{\convergenceOscNew};

    \addplot[MDSLstyle] table[%
    y={tTDSL},
    x={e1_MDSL},
    ]{\convergenceOscNew};

    \addplot[MFSLstyle] table[%
    y={tMFSL},
    x={e1_MFSL},
    ]{\convergenceOscNew};

    \addplot[Mstyle] table[%
    y={tM},
    x={e1_M},
    ]{\convergenceOscNew};

\end{axis}

\begin{axis}[%
    name=plot2,
    at=(plot1.below south west), anchor=above north west,
    width=0.95\linewidth,
    height=0.15\textheight,
    axis x line=bottom,
    ymode = log,
    xmode = log,
    log basis y={2},
    log basis x={2},
    axis y line=left,
    title={$\omega=10$, $\sigma=0.3$},
    ylabel={Wall-clock time $[s]$},
    xlabel={$E\|X(t_N)-Y_{N}\|$}]

    \addplot[TDSLstyle] table[%
    y={tTDSL},
    x={e1_TDSL},
    ]{\convergenceMedOscNew};

    \addplot[TFSLstyle] table[%
    y={tTFSL},
    x={e1_TFSL},
    ]{\convergenceMedOscNew};

    \addplot[MDSLstyle] table[%
    y={tMDSL},
    x={e1_MDSL},
    ]{\convergenceMedOscNew};

    \addplot[MFSLstyle] table[%
    y={tMFSL},
    x={e1_MFSL},
    ]{\convergenceMedOscNew};

    \addplot[Mstyle] table[%
    y={tM},
    x={e1_M},
    ]{\convergenceMedOscNew};

\end{axis}

\begin{axis}[%
    name=plot3,
    at=(plot2.below south west), anchor=above north west,
    width=0.95\linewidth,
    height=0.15\textheight,
    axis x line=bottom,
    ymode = log,
    xmode = log,
    log basis y={2},
    log basis x={2},
    axis y line=left,
    title={$\omega=50$, $\sigma=0.3$},
    ylabel={Wall-clock time $[s]$},
    xlabel={$E\|X(t_N)-Y_{N}\|$}]

    \addplot[TDSLstyle] table[%
    y={tTDSL},
    x={e1_TDSL},
    ]{\convergenceNew};

    \addplot[TFSLstyle] table[%
    y={tTFSL},
    x={e1_TFSL},
    ]{\convergenceNew};

    \addplot[MDSLstyle] table[%
    y={tMDSL},
    x={e1_MDSL},
    ]{\convergenceNew};

    \addplot[MFSLstyle] table[%
    y={tMFSL},
    x={e1_MFSL},
    ]{\convergenceNew};

    \addplot[Mstyle] table[%
    y={tM},
    x={e1_M},
    ]{\convergenceNew};

\end{axis}

\end{tikzpicture}
\caption{\mbox{}}
\end{subfigure}%
\caption{{The Kubo oscillator: Wall-clock time per batch of 25 paths vs.\ accuracy for different values of $\sigma$ and
$\omega$. }}\label{fig:EfficiencyKubo}
\end{figure}

The case in which $\sigma=\omega$ is presented in \cref{fig:KuboConvHH}. Again, there are no
significant differences between the methods for $\sigma=\omega=1$ and the {strong}
order is 0.5, as expected. For higher values of $\sigma$ and $\omega$, the advantage of the
FSL methods for larger step sizes is evident. In the cases of the midpoint and the DSL
schemes, the error is dominated by the linear stochastic diffusion term, thus causing the first
order behaviour. This term is incorporated in the exponentials of the FSL schemes, and will there thus not
contribute to the errors. For problems with small linear noise contributions
in \cref{fig:KuboConv}, all the SL schemes outperform the midpoint scheme for the
oscillatory cases, when $\omega\in\{10,50\}$.

For step sizes larger {than} $h=2^{-8}$, the Newton
solver frequently failed to solve the underlying nonlinear algebraic equations, in particular for
the TDSL scheme.
{The computational efforts depicted in \cref{fig:EfficiencyKubo} again demonstrate that there
are only small  differences in execution time between the methods, thus the more accurate methods are
also the most efficient in terms of wall-clock terms vs. accuracy.}

\pgfplotstableread[col sep=comma]{Data/KuboWeak/Non-linearKubooscillator_OuterIter=1_Miter=20_Mbatch=2500_TDSL_TFSL_M_MDSL_MFSL_tend=1_Weak_Param=1_1TrajectoryMax.csv}\convergenceOsc
\pgfplotstableread[col sep=comma]{Data/KuboWeak/Non-linearKubooscillator_OuterIter=1_Miter=20_Mbatch=2500_TDSL_TFSL_M_MDSL_MFSL_tend=1_Weak_Param=5_5TrajectoryMax.csv}\convergenceMedOsc
\pgfplotstableread[col sep=comma]{Data/KuboWeak/Non-linearKubooscillator_OuterIter=1_Miter=20_Mbatch=2500_TDSL_TFSL_M_MDSL_MFSL_tend=1_Weak_Param=10_10TrajectoryMax.csv}\convergence

\pgfplotstableread[col sep=comma]{Data/KuboWeak/Non-linearKubooscillator_OuterIter=1_Miter=20_Mbatch=2500_TDSL_TFSL_M_MDSL_MFSL_tend=1_Weak_Param=1_0.3TrajectoryMax.csv}\convergenceOscNew
\pgfplotstableread[col sep=comma]{Data/KuboWeak/Non-linearKubooscillator_OuterIter=1_Miter=20_Mbatch=2500_TDSL_TFSL_M_MDSL_MFSL_tend=1_Weak_Param=10_0.3TrajectoryMax.csv}\convergenceMedOscNew
\pgfplotstableread[col sep=comma]{Data/KuboWeak/Non-linearKubooscillator_OuterIter=1_Miter=20_Mbatch=2500_TDSL_TFSL_M_MDSL_MFSL_tend=1_Weak_Param=50_0.3TrajectoryMax.csv}\convergenceNew

\begin{figure}[ht!]
\centering
\begin{tikzpicture}
\centering
\begin{customlegend}[legend columns=5,legend style={align=left,draw=none,column sep=2ex},legend entries={TDSL, TFSL, MDSL, MFSL, Midpoint}]
        \addlegendimage{TDSLstyle}
        \addlegendimage{TFSLstyle}
\end{customlegend}
\end{tikzpicture}
\centering
\begin{tikzpicture}
\begin{customlegend}[legend columns=2,legend style={align=left,draw=none,column sep=2ex},legend entries={reference line with slope 1}]
        \addlegendimage{dotted}
\end{customlegend}
\end{tikzpicture}

\begin{subfigure}[b]{0.5\linewidth}
\centering\captionsetup{width=.8\linewidth}%
\begin{tikzpicture}
\begin{axis}[%
    name=plot1,
    width=0.95\linewidth,
    height=0.15\textheight,
    axis x line=bottom,
    ymode = log,
    xmode = log,
    log basis y={2},
    log basis x={2},
    xtick ={data},
    axis y line=left,
    title={$\omega=\sigma=1$},
    xlabel={$h$},
    ylabel={$|E(I(Y_{N})-1)|$}]

    \addplot[TDSLstyle] table[%
    x={h},
    y={e1_TDSL},
    ]{\convergenceOsc};

    \addplot[TFSLstyle] table[%
    x={h},
    y={e1_TFSL},
    ]{\convergenceOsc};
\addplot [domain=0.0000305:0.0039,dotted] {x/2};
\end{axis}

\begin{axis}[%
    name=plot2,
    at=(plot1.below south west), anchor=above north west,
    width=0.95\linewidth,
    height=0.15\textheight,
    axis x line=bottom,
    ymode = log,
    xmode = log,
    log basis y={2},
    log basis x={2},
    xtick ={data},
    axis y line=left,
    title={$\omega=\sigma=5$},
    xlabel={$h$},
    ylabel={$|E(I(Y_{N})-1)|$}]

    \addplot[TDSLstyle] table[%
    x={h},
    y={e1_TDSL},
    ]{\convergenceMedOsc};

    \addplot[TFSLstyle] table[%
    x={h},
    y={e1_TFSL},
    ]{\convergenceMedOsc};

\addplot [domain=0.0000305:0.0039,dotted] {x/2};

\end{axis}

\begin{axis}[%
    name=plot3,
    at=(plot2.below south west), anchor=above north west,
    width=0.95\linewidth,
    height=0.15\textheight,
    axis x line=bottom,
    ymode = log,
    xmode = log,
    log basis y={2},
    log basis x={2},
    xtick ={data},
    axis y line=left,
    title={$\omega=\sigma=10$},
    xlabel={$h$},
    ylabel={$|E(I(Y_{N})-1)|$}]

    \addplot[TDSLstyle] table[%
    x={h},
    y={e1_TDSL},
    ]{\convergence};

    \addplot[TFSLstyle] table[%
    x={h},
    y={e1_TFSL},
    ]{\convergence};
    \addplot [domain=0.0000305:0.0039,dotted] {2*x};

\end{axis}
\end{tikzpicture}
\caption{\mbox{}}
\end{subfigure}%
\hspace*{\fill}
\begin{subfigure}[b]{0.5\linewidth}
\centering\captionsetup{width=.8\linewidth}%
\begin{tikzpicture}
\begin{axis}[%
    name=plot1,
    width=0.95\linewidth,
    height=0.15\textheight,
    axis x line=bottom,
    ymode = log,
    xmode = log,
    log basis y={2},
    log basis x={2},
    xtick ={data},
    axis y line=left,
    title={$\omega=1$, $\sigma=0.3$},
    xlabel={$h$},
    ylabel={$|E(I(Y_{N})-1)|$}]

    \addplot[TDSLstyle] table[%
    x={h},
    y={e1_TDSL},
    ]{\convergenceOscNew};

    \addplot[TFSLstyle] table[%
    x={h},
    y={e1_TFSL},
    ]{\convergenceOscNew};
\addplot [domain=0.0000305:0.0039,dotted] {x/2};
\end{axis}

\begin{axis}[%
    name=plot2,
    at=(plot1.below south west), anchor=above north west,
    width=0.95\linewidth,
    height=0.15\textheight,
    axis x line=bottom,
    ymode = log,
    xmode = log,
    log basis y={2},
    log basis x={2},
    xtick ={data},
    axis y line=left,
    title={$\omega=10$, $\sigma=0.3$},
    xlabel={$h$},
    ylabel={$|E(I(Y_{N})-1)|$}]

    \addplot[TDSLstyle] table[%
    x={h},
    y={e1_TDSL},
    ]{\convergenceMedOscNew};

    \addplot[TFSLstyle] table[%
    x={h},
    y={e1_TFSL},
    ]{\convergenceMedOscNew};

\addplot [domain=0.0000305:0.0039,dotted] {x};

\end{axis}

\begin{axis}[%
    name=plot3,
    at=(plot2.below south west), anchor=above north west,
    width=0.95\linewidth,
    height=0.15\textheight,
    axis x line=bottom,
    ymode = log,
    xmode = log,
    log basis y={2},
    log basis x={2},
    xtick ={data},
    axis y line=left,
    title={$\omega=50$, $\sigma=0.3$},
    xlabel={$h$},
    ylabel={$|E(I(Y_{N})-1)|$}]

    \addplot[TDSLstyle] table[%
    x={h},
    y={e1_TDSL},
    ]{\convergenceNew};

    \addplot[TFSLstyle] table[%
    x={h},
    y={e1_TFSL},
    ]{\convergenceNew};

\addplot [domain=0.0000305:0.0039,dotted] {x};

\end{axis}

\end{tikzpicture}
\caption{\mbox{}}
\end{subfigure}%
\caption{{The Kubo oscillator: Weak error of $I$ vs.\ step size for different values of $\sigma$ and
$\omega$.}}\label{fig:KuboWeakConv}
\end{figure}

We next investigate how well the numerical solution preserves quadratic invariants.
\cref{ass:ABskew,ass:SkewSym} are satisfied, and by \cref{lem:InvSDE}  the exact solution of
\cref{equ:NonLinKuboa} is norm-preserving, thus
\begin{equation}
    \mathcal{I}(X(t)) = X(t)^{\top} X(t)
\end{equation}
is constant, and the solution will  stay on a circle.
The matrices $A_0$, $A_1$ and $A_2$ trivially commute; thus \cref{cor:QInvMidpoint} applies and
all the midpoint based methods are expected to preserve the invariant. As $g_2(x)\neq 0$,  \cref{cor:TrapzruleQinvariant}
no longer applies. {We thus expect to see a weak order one for the two trapezoidal SL schemes,
which is exactly what is observed in \cref{fig:KuboWeakConv}.}

{Finally, we want to study how well the quadratic invariants are preserved over time.}
Using each of the schemes, one solution path is calculated on the interval $[0,50]$, using step size
  $h = 2^{-5}$. The following two sets of parameters are chosen:  $\omega=\sigma=10$ and $\omega=100$
and $\sigma=0.3$. The values of $\I(Y_n)$ are shown in \cref{fig:kuboInv}. Again, for visibility,
only every 40th step is plotted.

\pgfplotstableread[col sep=comma]{Data/Kubo/h2-5Omega10Sigma10Start.csv}\MultiDKubo
\pgfplotstableread[col sep=comma]{Data/Kubo/h2-5Omega50Sigma03Start.csv}\MultiDKuboCohen

\begin{figure}[ht!]
\centering
\begin{tikzpicture}
\begin{customlegend}[legend columns=3,legend style={align=left,draw=none,column sep=1ex},legend entries={TDSL\\TFSL\\{MDSL, MFSL, Midpoint {(not distinguishable on these plots)}}\\}]
        \addlegendimage{TDSLstyle}
        \addlegendimage{TFSLstyle}
        \addlegendimage{legend image with MDSL MFSL Midpoint overlaid}
\end{customlegend}
\end{tikzpicture}
\begin{tikzpicture}
\begin{semilogyaxis}[%
    name=plot1,
    width=0.9\linewidth,
    height=0.15\textheight,
    axis x line=bottom,
    axis y line=left,
    ymax=100,
    title={$\omega=\sigma=10$},
    xlabel={$t_n$},
    ylabel={$\mathcal{I}(Y_n)$}]

    \addplot[TDSLstyle] table[x={t}, y expr={(\thisrow{TDSL1}*\thisrow{TDSL1} + \thisrow{TDSL2}*\thisrow{TDSL2})}]{\MultiDKubo};
    \addplot[TFSLstyle] table[x={t}, y expr={(\thisrow{TFSL1}*\thisrow{TFSL1} + \thisrow{TFSL2}*\thisrow{TFSL2})}]{\MultiDKubo};
    \addplot[MDSLstyle] table[x={t}, y expr={(\thisrow{MDSL1}*\thisrow{MDSL1} + \thisrow{MDSL2}*\thisrow{MDSL2})}]{\MultiDKubo};
    \addplot[MFSLstyle] table[x={t}, y expr={(\thisrow{MFSL1}*\thisrow{MFSL1} + \thisrow{MFSL2}*\thisrow{MFSL2})}]{\MultiDKubo};
    \addplot[Mstyle] table[x={t}, y expr={(\thisrow{Midpoint1}*\thisrow{Midpoint1} + \thisrow{Midpoint2}*\thisrow{Midpoint2})}]{\MultiDKubo};
\end{semilogyaxis}

\begin{axis}[%
    name=plot2,
    at=(plot1.below south west), anchor=above north west,
    width=0.9\linewidth,
    height=0.15\textheight,
    axis x line=bottom,
    axis y line=left,
    title={$\omega=50$, $\sigma=0.3$},
    xlabel={$t_n$},
    ylabel={$\mathcal{I}(Y_n)$},
    scaled y ticks = false,
    log ticks with fixed point,
    yticklabel style={/pgf/number format/.cd,fixed,precision=1, fixed zerofill},]

    \addplot[TDSLstyle] table[x={t}, y expr={(\thisrow{TDSL1}*\thisrow{TDSL1} + \thisrow{TDSL2}*\thisrow{TDSL2})}]{\MultiDKuboCohen};
    \addplot[TFSLstyle] table[x={t}, y expr={(\thisrow{TFSL1}*\thisrow{TFSL1} + \thisrow{TFSL2}*\thisrow{TFSL2})}]{\MultiDKuboCohen};
    \addplot[MDSLstyle] table[x={t}, y expr={(\thisrow{MDSL1}*\thisrow{MDSL1} + \thisrow{MDSL2}*\thisrow{MDSL2})}]{\MultiDKuboCohen};
    \addplot[MFSLstyle] table[x={t}, y expr={(\thisrow{MFSL1}*\thisrow{MFSL1} + \thisrow{MFSL2}*\thisrow{MFSL2})}]{\MultiDKuboCohen};
    \addplot[Mstyle] table[x={t}, y expr={(\thisrow{Midpoint1}*\thisrow{Midpoint1} + \thisrow{Midpoint2}*\thisrow{Midpoint2})}]{\MultiDKuboCohen};
\end{axis}
\end{tikzpicture}

\caption{The non-linear Kubo oscillator: Evaluation of the invariant $\I(Y_n)$} \label{fig:kuboInv}
\end{figure}

In both cases, the mid-point based schemes preserve the invariant, while there is a significant
drift-off for the two trapezoidal based schemes. The case $\omega=\sigma=10$ is the example
depicted in \cref{fig:IntroExample}, although there, the integration interval is only $[0,1]$.
In \cref{fig:kuboInv}, where the integration is done over longer time, the drift-off leaves the TDSL
scheme basically useless. The midpoint based methods preserve the invariant.
In the low noise case, the drift-off is almost the same for the two
trapezoidal schemes, and again, there is none or little drift-off in the midpoint schemes.

\subsection{Stochastic Fermi-Pasta-Ulam-Tsingou problem}
Finally, we would like to test the SL schemes by applying them to a more complex highly oscillatory
problem. To this aim, a stochastic modification of the famous deterministic Fermi-Pasta-Ulam-Tsingou (FPUT) problem as described in
\citet[Chapter I.5]{hairer06gni} is considered. The deterministic problem is defined by the Hamiltonian

\begin{equation}\label{equ:FPUTHamil}
\begin{multlined}
    \mathcal{I}_0(x,y) = \sum_{m=1}^{M} J_m  + \frac{1}{4} \left((x_{0,1}-x_{1,1})^4 +
    (x_{0,M}+x_{1,M})^4 \right) \\
    + \frac{1}{4}\sum_{m=1}^{M-1} \left(x_{0,m+1}-x_{1,m+1}-x_{0,m}-x_{1,m}\right)^4
    \end{multlined}
\end{equation}
where
\[
J_m=\frac{1}{2} \left(y_{0,m}^2 + y_{1,m}^2\right)+
    \frac{\omega^2}{2}  x_{1,m}^2.
\]
In our example, stochastic noise of strength $\sigma_mJ_m$ is added to each of the springs. The resulting system is a $4M$
dimensional SDE:
\begin{align*}
  dx_{0,m} &= y_{0,m}\dt + \sigma_m y_{0,m}\circ \dW_m,  \\
  dx_{1,m} &= y_{1,m} \dt + \sigma_m y_{1,m}\circ \dW_m, \\
  dy_{0,m} &= (- g_m + g_{m+1})\dt, \\
  dy_{1,m} &= (-\omega^2 x_{1,m}  + g_{m} + g_{m+1})\dt -\sigma_m \omega^2 x_{1,m} \circ \dW_m,
\end{align*}
for $m=1,\ldots,M$  where
\begin{align*}
  g_1 &= (x_{0,1}-x_{1,1})^3,\\
  g_m &= (x_{0,m}-x_{1,m}-x_{0,m-1}-x_{1,m-1})^3, \qquad m=2,\ldots,M, \\
  g_{M+1} &= -(x_{0,M}+x_{1,M})^3.
\end{align*}
Let $X={\big(}(x_{0,m})_{m=1}^M,(x_{1,m})_{m=0}^M,(y_{0,m})_{m=1}^M,(y_{1,m})_{m=0}^M{\big)}^\top$.
The complete system can be written as a semi-linear SDE
\begin{equation} \label[SDE]{equ:FPU_SDE}
  \dX = (A_0+g_0)\dt + \sum_{m=1}^M A_m X \circ \dW_m
\end{equation}
where $A_m$ are $4M\times4M$ block matrices given by
\begin{equation}
  A_m = \begin{pmatrix}
  0 & 0 & D_m & 0 \\
  0 & 0 & 0 & D_m \\
  0 & 0 & 0 & 0 \\
  0 & -\omega^2D_m & 0 & 0
  \end{pmatrix}
  \label{equ:fpu_matrix}
\end{equation}
where $D_m = \sigma_m
\text{diag}\{0,\dotsc,1,\dotsc,0\}$, with the nonzero element in position $m$, for $m=1,\ldots,M$, and $D_0=I_M$, the identity matrix of dimension $M\times M$.
The matrices $A_m$ satisfy \cref{ass:commute}, so the SL schemes are applicable. {For the same
reason, the expected {strong} order is 1 for all methods considered.}
\cref{ass:ABskew,ass:SkewSym} are not satisfied, so there is no quadratic invariant to preserve.

For the simulations we have used $M=3$, $\omega=50$, $\sigma_m$ {$=\sigma$}
for all $m$, {with $\sigma\in \{0.02, 0.2\}$}, and
initial values
\[ x_{0,1}(0)=y_{0,1}(0)=1, \quad x_{1,1}(0)=\omega^{-1}, \quad  y_{1,1}(0)=1.  \]
The remaining initial values are all zero.
Once again, we measure the {strong} error based on 1000 simulations. The SDE is solved over the interval $[0,1]$, and the
reference solutions are computed by {the} MFSL scheme with $h_{\text{ref}}=2^{-17}$. The results are
presented in \cref{fig:ConvFPUT}. The $95\%$-confidence intervals have been calculated and in all
cases they span less than {$14\%$} of the shown mean values.

\pgfplotstableread[col sep=comma]{Data/FPUT/Fermi-Pasta-Ulam-Tsingouproblem-Anne_OuterIter=1_Miter=40_Mbatch=25_Nh=17_TDSL_TFSL_M_MDSL_MFSL_tend=1_Strong_Param=50_0.02Endpoint.csv}\ConvFPUT

\pgfplotstableread[col sep=comma]{Data/FPUT/Fermi-Pasta-Ulam-Tsingouproblem-Anne_OuterIter=1_Miter=40_Mbatch=25_Nh=17_TDSL_TFSL_M_MDSL_MFSL_tend=1_Strong_Param=50_0.2Endpoint.csv}\ConvFPUTSigmanullto

\begin{figure}[ht!]
\begin{tikzpicture}

\begin{customlegend}[legend columns=5,legend style={align=left,draw=none,column sep=2ex},legend entries={TDSL, TFSL, MDSL, MFSL, Midpoint}]
        \addlegendimage{TDSLstyle}
        \addlegendimage{TFSLstyle}
        \addlegendimage{MDSLstyle}
        \addlegendimage{MFSLstyle}
        \addlegendimage{Mstyle}
\end{customlegend}
\end{tikzpicture}
\centering
\begin{tikzpicture}
\begin{customlegend}[legend columns=1,legend style={align=left,draw=none,column sep=2ex},legend entries={{reference line with slope 1}}]
        \addlegendimage{dotted}
\end{customlegend}
\end{tikzpicture}
\begin{subfigure}[b]{0.5\linewidth}
\centering
\begin{tikzpicture}
\begin{axis}[%
    name=plot1,
    width=0.9\textwidth,
    height=0.15\textheight,
    axis x line=bottom,
    ymode = log,
    xmode = log,
    log basis y={2},
    log basis x={2},
    xtick ={data},
    ytick ={0.5,0.0625,0.0078125,0.0009765625,0.0001220703125,0.00001525878906},
    axis y line=left,
    xlabel={$h$},
    ylabel={$E\|X(T)-Y_{N}\|$}]

    \addplot[TDSLstyle] table[%
    x={h},
    y={e1_TDSL},
    ]{\ConvFPUT};

    \addplot[TFSLstyle] table[%
    x={h},
    y={e1_TFSL},
    ]{\ConvFPUT};

    \addplot[MDSLstyle] table[%
    x={h},
    y={e1_MDSL},
    ]{\ConvFPUT};

    \addplot[MFSLstyle] table[%
    x={h},
    y={e1_MFSL},
    ]{\ConvFPUT};

    \addplot[Mstyle] table[%
    x={h},
    y={e1_M},
    ]{\ConvFPUT};

    \addplot [domain=0.00048828:0.1250,dotted] {2*x};

\end{axis}

\end{tikzpicture}
\caption{{$\sigma=0.02$}}
\end{subfigure}%
\begin{subfigure}[b]{0.5\linewidth}
\begin{tikzpicture}
\begin{axis}[%
    name=plot2,
    width=0.9\textwidth,
    height=0.15\textheight,
    axis x line=bottom,
    ymode = log,
    xmode = log,
    log basis y={2},
    log basis x={2},
    xtick ={data},
    ytick ={0.5,0.0625,0.0078125,0.0009765625,0.0001220703125,0.00001525878906},
    axis y line=left,
    xlabel={{$h$}},
    ylabel={{$E\|X(T)-Y_{N}\|$}}]

    \addplot[TDSLstyle] table[%
    x={h},
    y={e1_TDSL},
    ]{\ConvFPUTSigmanullto};

    \addplot[TFSLstyle] table[%
    x={h},
    y={e1_TFSL},
    ]{\ConvFPUTSigmanullto};

    \addplot[MDSLstyle] table[%
    x={h},
    y={e1_MDSL},
    ]{\ConvFPUTSigmanullto};

    \addplot[MFSLstyle] table[%
    x={h},
    y={e1_MFSL},
    ]{\ConvFPUTSigmanullto};

    \addplot[Mstyle] table[%
    x={h},
    y={e1_M},
    ]{\ConvFPUTSigmanullto};

    \addplot [domain=0.00048828:0.1250,dotted] {4*x};

\end{axis}

\end{tikzpicture}
\caption{{$\sigma=0.2$}}
\end{subfigure}
    \caption{Convergence of the TDSL, TFSL, MDSL, MFSL and implicit midpoint rule for the stochastic Fermi-Pasta-Ulam-Tsingou problem} \label{fig:ConvFPUT}
\end{figure}

{In all cases the Lawson schemes perform significantly better than the implicit midpoint rule.
For small values of $\sigma$ the two FSL schemes are only slightly better than the two DSL
schemes. For larger values of $\sigma$, the FSL schemes outperform the DSL
schemes. This is in line with what was observed for the Kubo oscillator in \cref{fig:KuboConv}. But
even in the small noise case,} the feasible step sizes of the Lawson schemes are larger
than the feasible step sizes for the implicit midpoint rule - the convergence deteriorates for the
implicit midpoint rule approximately when $h=2^{-6}$, whereas the convergence of the Lawson schemes
only deteriorates around $h=2^{-3}$, the SL schemes thus allowing to use much larger step sizes than
the implicit midpoint rule. {For larger values of $\sigma$, the midpoint rule is in practice useless
for the step sizes applied here.}

\begin{figure}[ht!]
\centering
\begin{tikzpicture}

\begin{customlegend}[legend columns=5,legend style={align=left,draw=none,column sep=2ex},legend entries={TDSL, TFSL, MDSL, MFSL, Midpoint}]
        \addlegendimage{TDSLstyle}
        \addlegendimage{TFSLstyle}
        \addlegendimage{MDSLstyle}
        \addlegendimage{MFSLstyle}
        \addlegendimage{Mstyle}
\end{customlegend}
\end{tikzpicture}
\begin{subfigure}[b]{0.5\linewidth}
\centering
\begin{tikzpicture}
\begin{axis}[%
    name=plot1,
    width=0.9\textwidth,
    height=0.15\textheight,
    axis x line=bottom,
    ymode = log,
    xmode = log,
    log basis y={2},
    log basis x={2},
    xtick ={0.5,0.0625,0.0078125,0.0009765625,0.0001220703125,0.00001525878906},
    axis y line=left,
    ylabel={Wall-clock time $[s]$},
    xlabel={$E\|X(T)-Y_{N}\|$}]

    \addplot[TDSLstyle] table[%
    y={tTDSL},
    x={e1_TDSL},
    ]{\ConvFPUT};

    \addplot[TFSLstyle] table[%
    y={tTFSL},
    x={e1_TFSL},
    ]{\ConvFPUT};

    \addplot[MDSLstyle] table[%
    y={tMDSL},
    x={e1_MDSL},
    ]{\ConvFPUT};

    \addplot[MFSLstyle] table[%
    y={tMFSL},
    x={e1_MFSL},
    ]{\ConvFPUT};

    \addplot[Mstyle] table[%
    y={tM},
    x={e1_M},
    ]{\ConvFPUT};

\end{axis}

\end{tikzpicture}
\caption{$\sigma=0.02$}
\end{subfigure}%
\begin{subfigure}[b]{0.5\linewidth}
\centering
\begin{tikzpicture}
\begin{axis}[%
    name=plot1,
    width=0.9\textwidth,
    height=0.15\textheight,
    axis x line=bottom,
    ymode = log,
    xmode = log,
    log basis y={2},
    log basis x={2},
    xtick ={0.5,0.0625,0.0078125,0.0009765625,0.0001220703125,0.00001525878906},
    axis y line=left,
    ylabel={Wall-clock time $[s]$},
    xlabel={$E\|X(T)-Y_{N}\|$}]

    \addplot[TDSLstyle] table[%
    y={tTDSL},
    x={e1_TDSL},
    ]{\ConvFPUTSigmanullto};

    \addplot[TFSLstyle] table[%
    y={tTFSL},
    x={e1_TFSL},
    ]{\ConvFPUTSigmanullto};

    \addplot[MDSLstyle] table[%
    y={tMDSL},
    x={e1_MDSL},
    ]{\ConvFPUTSigmanullto};

    \addplot[MFSLstyle] table[%
    y={tMFSL},
    x={e1_MFSL},
    ]{\ConvFPUTSigmanullto};

    \addplot[Mstyle] table[%
    y={tM},
    x={e1_M},
    ]{\ConvFPUTSigmanullto};

\end{axis}

\end{tikzpicture}
\caption{$\sigma=0.2$}
\end{subfigure}
    \caption{{Wall-clock time per batch of 25 paths vs.\ accuracy of the TDSL, TFSL, MDSL, MFSL and implicit midpoint rule for the stochastic Fermi-Pasta-Ulam-Tsingou problem}} \label{fig:EfficiencyFPUT}
\end{figure}
{The work-precision diagram \cref{fig:EfficiencyFPUT} again demonstrates that even for this
  slighly larger problem the higher accuracy of the SL schemes, in particular the
  FSL schemes, {compensates} for the disadvantage of the additional computational work required for the
calculation of the matrix exponentials. In the current example, the TFSL method is slightly the most
efficient.}

\pgfplotstableread[col sep=comma]{Data/FPUT/Fermi-Pasta-Ulam-Tsingouproblem-Anne_OuterIter=1_Miter=40_Mbatch=1250_Nh=17_TDSL_TFSL_M_MDSL_MFSL_tend=1_Weak_Param=50_0.02TrajectoryMax.csv}\ConvFPUT

\pgfplotstableread[col
sep=comma]{Data/FPUT/Fermi-Pasta-Ulam-Tsingouproblem-Anne_OuterIter=80_Miter=20_Mbatch=1250_Nh=17_TDSL_TFSL_M_MDSL_MFSL_tend=1_Weak_Param=50_0.2TrajectoryMax.csv}\ConvFPUTb
\begin{figure}[ht!]
\centering
\begin{tikzpicture}

\begin{customlegend}[legend columns=5,legend style={align=left,draw=none,column sep=2ex},legend entries={TDSL, TFSL, MDSL, MFSL, Midpoint}]
        \addlegendimage{TDSLstyle}
        \addlegendimage{TFSLstyle}
        \addlegendimage{MDSLstyle}
        \addlegendimage{MFSLstyle}
        \addlegendimage{Mstyle}
\end{customlegend}
\end{tikzpicture}
\centering
\begin{tikzpicture}
\begin{customlegend}[legend columns=2,legend style={align=left,draw=none,column sep=2ex},legend entries={reference line with slope 1, reference line with slope 2}]
\addlegendimage{dotted}
        \addlegendimage{dashdotted}
\end{customlegend}
\end{tikzpicture}
\begin{subfigure}[b]{0.5\linewidth}
\centering\captionsetup{width=.8\linewidth}%
\begin{tikzpicture}
\begin{axis}[%
    name=plot1,
    width=0.9\textwidth,
    height=0.15\textheight,
    axis x line=bottom,
    ymode = log,
    xmode = log,
    log basis y={2},
    log basis x={2},
    xtick ={data},
    axis y line=left,
    xlabel={$h$},
    ylabel={$|E(I_1(X(T))-I_1(Y_{N}))|$}]

    \addplot[TDSLstyle] table[%
    x={h},
    y={e1_TDSL},
    ]{\ConvFPUT};

    \addplot[TFSLstyle] table[%
    x={h},
    y={e1_TFSL},
    ]{\ConvFPUT};

    \addplot[MDSLstyle] table[%
    x={h},
    y={e1_MDSL},
    ]{\ConvFPUT};

    \addplot[MFSLstyle] table[%
    x={h},
    y={e1_MFSL},
    ]{\ConvFPUT};

    \addplot[Mstyle] table[%
    x={h},
    y={e1_M},
    ]{\ConvFPUT};

    \addplot [domain=0.00048828:0.1250,dotted] {2*x};
    \addplot [domain=0.00048828:0.1250,dashdotted] {32*x*x};

\end{axis}

\begin{axis}[%
    name=plot2,
    at=(plot1.below south west), anchor=above north west,
    width=0.9\textwidth,
    height=0.15\textheight,
    axis x line=bottom,
    ymode = log,
    xmode = log,
    log basis y={2},
    log basis x={2},
    xtick ={data},
    axis y line=left,
    xlabel={$h$},
    ylabel={$|E(I_2(X(T))-I_2(Y_{N}))|$}]

    \addplot[TDSLstyle] table[%
    x={h},
    y={e2_TDSL},
    ]{\ConvFPUT};

    \addplot[TFSLstyle] table[%
    x={h},
    y={e2_TFSL},
    ]{\ConvFPUT};

    \addplot[MDSLstyle] table[%
    x={h},
    y={e2_MDSL},
    ]{\ConvFPUT};

    \addplot[MFSLstyle] table[%
    x={h},
    y={e2_MFSL},
    ]{\ConvFPUT};

    \addplot[Mstyle] table[%
    x={h},
    y={e2_M},
    ]{\ConvFPUT};

    \addplot [domain=0.00048828:0.1250,dotted] {2*x};
    \addplot [domain=0.00048828:0.1250,dashdotted] {32*x*x};

\end{axis}

\begin{axis}[%
    name=plot3,
    at=(plot2.below south west), anchor=above north west,
    width=0.9\textwidth,
    height=0.15\textheight,
    axis x line=bottom,
    ymode = log,
    xmode = log,
    log basis y={2},
    log basis x={2},
    xtick ={data},
    axis y line=left,
    xlabel={$h$},
    ylabel={$|E(I_3(X(T))-I_3(Y_{N}))|$}]

    \addplot[TDSLstyle] table[%
    x={h},
    y={e3_TDSL},
    ]{\ConvFPUT};

    \addplot[TFSLstyle] table[%
    x={h},
    y={e3_TFSL},
    ]{\ConvFPUT};

    \addplot[MDSLstyle] table[%
    x={h},
    y={e3_MDSL},
    ]{\ConvFPUT};

    \addplot[MFSLstyle] table[%
    x={h},
    y={e3_MFSL},
    ]{\ConvFPUT};

    \addplot[Mstyle] table[%
    x={h},
    y={e3_M},
    ]{\ConvFPUT};

    \addplot [domain=0.00048828:0.1250,dotted] {x/128};
    \addplot [domain=0.00048828:0.1250,dashdotted] {x*x/16};

\end{axis}

\begin{axis}[%
    name=plot4,
    at=(plot3.below south west), anchor=above north west,
    width=0.9\textwidth,
    height=0.15\textheight,
    axis x line=bottom,
    ymode = log,
    xmode = log,
    log basis y={2},
    log basis x={2},
    xtick ={data},
    axis y line=left,
    xlabel={$h$},
    ylabel={$|E(\It(X(T))-\It(Y_{N}))|$}]

    \addplot[TDSLstyle] table[%
    x={h},
    y={e4_TDSL},
    ]{\ConvFPUT};

    \addplot[TFSLstyle] table[%
    x={h},
    y={e4_TFSL},
    ]{\ConvFPUT};

    \addplot[MDSLstyle] table[%
    x={h},
    y={e4_MDSL},
    ]{\ConvFPUT};

    \addplot[MFSLstyle] table[%
    x={h},
    y={e4_MFSL},
    ]{\ConvFPUT};

    \addplot[Mstyle] table[%
    x={h},
    y={e4_M},
    ]{\ConvFPUT};

    \addplot [domain=0.00048828:0.1250,dotted] {2*x};
    \addplot [domain=0.00048828:0.1250,dashdotted] {32*x*x};

\end{axis}

\end{tikzpicture}
\caption{$\sigma=0.02$}
\end{subfigure}%
\hspace*{\fill}%
\begin{subfigure}[b]{0.5\linewidth}
\centering\captionsetup{width=.8\linewidth}%
\begin{tikzpicture}
\begin{axis}[%
    name=plot1,
    width=0.9\textwidth,
    height=0.15\textheight,
    axis x line=bottom,
    ymode = log,
    xmode = log,
    log basis y={2},
    log basis x={2},
    xtick ={data},
    axis y line=left,
    xlabel={$h$},
    ylabel={$|E(I_1(X(T))-I_1(Y_{N}))|$}]

    \addplot[TDSLstyle] table[%
    x={h},
    y={e1_TDSL},
    ]{\ConvFPUTb};

    \addplot[TFSLstyle] table[%
    x={h},
    y={e1_TFSL},
    ]{\ConvFPUTb};

    \addplot[MDSLstyle] table[%
    x={h},
    y={e1_MDSL},
    ]{\ConvFPUTb};

    \addplot[MFSLstyle] table[%
    x={h},
    y={e1_MFSL},
    ]{\ConvFPUTb};

    \addplot[Mstyle] table[%
    x={h},
    y={e1_M},
    ]{\ConvFPUTb};

    \addplot [domain=0.00048828:0.1250,dotted] {x/128};
    \addplot [domain=0.00048828:0.1250,dashdotted] {4*x*x};

\end{axis}
\begin{axis}[%
    name=plot2,
    at=(plot1.below south west), anchor=above north west,
    width=0.9\textwidth,
    height=0.15\textheight,
    axis x line=bottom,
    ymode = log,
    xmode = log,
    log basis y={2},
    log basis x={2},
    xtick ={data},
    axis y line=left,
    xlabel={$h$},
    ylabel={$|E(I_2(X(T))-I_2(Y_{N}))|$}]

    \addplot[TDSLstyle] table[%
    x={h},
    y={e2_TDSL},
    ]{\ConvFPUTb};

    \addplot[TFSLstyle] table[%
    x={h},
    y={e2_TFSL},
    ]{\ConvFPUTb};

    \addplot[MDSLstyle] table[%
    x={h},
    y={e2_MDSL},
    ]{\ConvFPUTb};

    \addplot[MFSLstyle] table[%
    x={h},
    y={e2_MFSL},
    ]{\ConvFPUTb};

    \addplot[Mstyle] table[%
    x={h},
    y={e2_M},
    ]{\ConvFPUTb};

    \addplot [domain=0.00048828:0.1250,dotted] {2*x};
    \addplot [domain=0.00048828:0.1250,dashdotted] {2*x*x};

\end{axis}

\begin{axis}[%
    name=plot3,
    at=(plot2.below south west), anchor=above north west,
    width=0.9\textwidth,
    height=0.15\textheight,
    axis x line=bottom,
    ymode = log,
    xmode = log,
    log basis y={2},
    log basis x={2},
    xtick ={data},
    axis y line=left,
    xlabel={$h$},
    ylabel={$|E(I_3(X(T))-I_3(Y_{N}))|$}]

    \addplot[TDSLstyle] table[%
    x={h},
    y={e3_TDSL},
    ]{\ConvFPUTb};

    \addplot[TFSLstyle] table[%
    x={h},
    y={e3_TFSL},
    ]{\ConvFPUTb};

    \addplot[MDSLstyle] table[%
    x={h},
    y={e3_MDSL},
    ]{\ConvFPUTb};

    \addplot[MFSLstyle] table[%
    x={h},
    y={e3_MFSL},
    ]{\ConvFPUTb};

    \addplot[Mstyle] table[%
    x={h},
    y={e3_M},
    ]{\ConvFPUTb};

    \addplot [domain=0.00048828:0.1250,dotted] {x/128};
    \addplot [domain=0.00048828:0.1250,dashdotted] {x*x/32};

\end{axis}

\begin{axis}[%
    name=plot4,
    at=(plot3.below south west), anchor=above north west,
    width=0.9\textwidth,
    height=0.15\textheight,
    axis x line=bottom,
    ymode = log,
    xmode = log,
    log basis y={2},
    log basis x={2},
    xtick ={data},
    axis y line=left,
    xlabel={$h$},
    ylabel={$|E(\It(X(T))-\It(Y_{N}))|$}]

    \addplot[TDSLstyle] table[%
    x={h},
    y={e4_TDSL},
    ]{\ConvFPUTb};

    \addplot[TFSLstyle] table[%
    x={h},
    y={e4_TFSL},
    ]{\ConvFPUTb};

    \addplot[MDSLstyle] table[%
    x={h},
    y={e4_MDSL},
    ]{\ConvFPUTb};

    \addplot[MFSLstyle] table[%
    x={h},
    y={e4_MFSL},
    ]{\ConvFPUTb};

    \addplot[Mstyle] table[%
    x={h},
    y={e4_M},
    ]{\ConvFPUTb};

    \addplot [domain=0.00048828:0.1250,dotted] {2*x};
    \addplot [domain=0.00048828:0.1250,dashdotted] {32*x*x};

\end{axis}

\end{tikzpicture}
\caption{$\sigma=0.2$}
\end{subfigure}
    \caption{{Weak convergence of $I_1$, $I_2$, $I_3$ and $\It=I_1+I_2+I_3$ for the TDSL, TFSL, MDSL, MFSL and implicit midpoint
    rule for the stochastic Fermi-Pasta-Ulam-Tsingou problem}} \label{fig:ConvFPUTWeak}
\end{figure}

{Finally, we consider the weak convergence of the methods when applied to calculate the total oscillatory energy $\It=\sum_{j=1}^{M}I_j$ where the oscillatory energy of the $j$-th string is given by $I_j=\frac12(y_{1,j}^2+\omega^2x_{1,j}^2)$ for the FPUT problem,
  and the result is presented in \cref{fig:ConvFPUTWeak}. In the small noise case, when
  $\sigma=0.02$, the error is clearly dominated by the diffusion term, and the order of the method
  is close to the deterministic order two. Still, the FSL schemes are significantly more accurate
  than the midpoint rule, although the difference between those again is minuscule. When the noise level
  increases to $\sigma=0.2$, the midpoint and the DSL schemes perform almost the same, while the
  errors of the FSL schemes are significantly smaller. We also observe the order two  behaviour of
  the trapezoidal FSL scheme, although, as already mentioned, no quadratic invariants are conserved
  in this case. However, in the case where $g_m=0$ for $m=1,\dotsc,M$, the transformed system
  \cref{equ:autonomousSDE} is in fact an SDE with additive noise, for which the weak order of the
  trapezoidal rule is two, see e.\,g.\ \citet{milstein04snf}. In \cref{fig:EfficiencyFPUTWeak} we observe how
  this makes the trapezoidal FSL scheme the most efficient what concerns the weak error.
}

\begin{figure}[ht!]
\centering
\begin{tikzpicture}

\begin{customlegend}[legend columns=5,legend style={align=left,draw=none,column sep=2ex},legend entries={TDSL, TFSL, MDSL, MFSL, Midpoint}]
        \addlegendimage{TDSLstyle}
        \addlegendimage{TFSLstyle}
        \addlegendimage{MDSLstyle}
        \addlegendimage{MFSLstyle}
        \addlegendimage{Mstyle}
\end{customlegend}
\end{tikzpicture}
\begin{subfigure}[b]{0.5\linewidth}
\centering
\begin{tikzpicture}
\begin{axis}[%
    name=plot1,
    width=0.9\textwidth,
    height=0.15\textheight,
    axis x line=bottom,
    ymode = log,
    xmode = log,
    log basis y={2},
    log basis x={2},
    axis y line=left,
    ylabel={Wall-clock time $[s]$},
    xlabel={$|E(\It(X(T))-\It(Y_{N}))|$}]

    \addplot[TDSLstyle] table[%
    y={tTDSL},
    x={e4_TDSL},
    ]{\ConvFPUT};

    \addplot[TFSLstyle] table[%
    y={tTFSL},
    x={e4_TFSL},
    ]{\ConvFPUT};

    \addplot[MDSLstyle] table[%
    y={tMDSL},
    x={e4_MDSL},
    ]{\ConvFPUT};

    \addplot[MFSLstyle] table[%
    y={tMFSL},
    x={e4_MFSL},
    ]{\ConvFPUT};

    \addplot[Mstyle] table[%
    y={tM},
    x={e4_M},
    ]{\ConvFPUT};
\end{axis}

\end{tikzpicture}
\caption{$\sigma=0.02$}
\end{subfigure}%
\hspace*{\fill}
\begin{subfigure}[b]{0.5\linewidth}
\centering
\begin{tikzpicture}
\begin{axis}[%
    name=plot1,
    width=0.9\textwidth,
    height=0.15\textheight,
    axis x line=bottom,
    ymode = log,
    xmode = log,
    log basis y={2},
    log basis x={2},
    axis y line=left,
    ylabel={Wall-clock time $[s]$},
    xlabel={$|E(\It(X(T))-\It(Y_{N}))|$}]

    \addplot[TDSLstyle] table[%
    y={tTDSL},
    x={e4_TDSL},
    ]{\ConvFPUTb};

    \addplot[TFSLstyle] table[%
    y={tTFSL},
    x={e4_TFSL},
    ]{\ConvFPUTb};

    \addplot[MDSLstyle] table[%
    y={tMDSL},
    x={e4_MDSL},
    ]{\ConvFPUTb};

    \addplot[MFSLstyle] table[%
    y={tMFSL},
    x={e4_MFSL},
    ]{\ConvFPUTb};

    \addplot[Mstyle] table[%
    y={tM},
    x={e4_M},
    ]{\ConvFPUTb};

\end{axis}

\end{tikzpicture}

\caption{$\sigma=0.2$}
\end{subfigure}
    \caption{{Wall-clock time per batch of 1250 paths vs.\ weak error of $\It=I_1+I_2+I_3$ for the TDSL, TFSL, MDSL, MFSL and implicit midpoint rule for the stochastic Fermi-Pasta-Ulam-Tsingou problem}} \label{fig:EfficiencyFPUTWeak}
\end{figure}

\section{Conclusion}
In this paper, we proved that stochastic Lawson schemes under suitable conditions preserve linear and quadratic invariants. We proved that the trapezoidal stochastic Lawson scheme nearly preserves quadratic invariants if the diffusion terms are linear and fully included in the exponential. These results have been verified by numerical experiments.

For stochastic differential equations with highly oscillatory drift and diffusion we numerically
demonstrated that full stochastic Lawson schemes allow for larger step-sizes than standard schemes,
{in the sense of being able to resolve high oscillations}.
{All methods are implicit, and in our implementations, the cpu-time used for solving the
  nonlinear systems was dominant compared to the time used to evaluate the matrix exponentials, and in terms of accuracy
versus computational work, the more accurate methods were the most efficient. This, of course, depends
heavily on the problem at hand, for larger problems with more demanding matrix
exponential evaluations, the drift Lawson methods might turn out to be preferable.}

\section*{Acknowledgements}
Nicky Cordua Mattsson would like to thank the SDU e-Science centre for partially funding his PhD and
the Department of Mathematics at the Norwegian University of Science and Technology for kindly
hosting him during his visit. {The authors would like to thank the unknown referees for their very helpful
comments.}

\def\cprime{$'$} \providecommand{\de}[2]{#2}

\end{document}